\documentclass{amsart}
\usepackage{amssymb}
\usepackage{amscd}
\usepackage{amsmath}
\usepackage{amsfonts}
\usepackage{latexsym}
\usepackage[all]{xy}
\usepackage{mathbbol}
\usepackage{hyperref}
\usepackage{verbatim}

\newtheorem{theorem}{Theorem}[section]
\newtheorem{lemma}[theorem]{Lemma}
\newtheorem{proposition}[theorem]{Proposition}
\newtheorem{corollary}[theorem]{Corollary}
\theoremstyle{definition}
\newtheorem{definition}[theorem]{Definition}
\newtheorem{example}[theorem]{Example}

\newtheorem{question}[theorem]{Question}
\newtheorem{remark}[theorem]{Remark}

\newcommand{\Tr}{\text{Tr}}
\newcommand{\id}{\text{id}}

\newcommand{\FPdim}{\text{FPdim}}

\newcommand{\Fun}{\text{Fun}}

\renewcommand{\Vec}{\text{Vec}}

\newcommand{\Hom}{\text{Hom}}

\newcommand{\Aut}{\text{Aut}}

\newcommand{\Rep}{\mathrm{Rep}}

\newcommand{\rev}{\text{rev}}

\newcommand{\op}{\text{op}}

\newcommand{\B}{\mathcal{B}}
\newcommand{\C}{\mathcal{C}}

\newcommand{\D}{\mathcal{D}}

\newcommand{\E}{\mathcal{E}}

\newcommand{\Z}{\mathcal{Z}}

\newcommand{\M}{\mathcal{M}}
\newcommand{\A}{\mathcal{A}}

\newcommand{\W}{\mathcal{W}}

\newcommand{\QM}{\mathcal{Q}\mathfrak{M}}
\newcommand{\FC}{\mathfrak{FC}}

\renewcommand{\O}{\mathcal{O}}

\newcommand{\be}{\mathbf{1}}
\newcommand{\fg}{{\mathfrak g}}

\newcommand{\fL}{{\mathfrak L}}

\newcommand{\g}{\mathfrak{g}}
\newcommand{\h}{\mathfrak{h}}

\renewcommand{\be}{\mathbf{1}}

\newcommand{\cS}{{\mathcal S}}

\newcommand{\CC}{{\mathbb{C}}}
\newcommand{\BR}{{\mathbb{R}}}
\newcommand{\BZ}{{\mathbb Z}}
\newcommand{\BQ}{{\mathbb Q}}

\newcommand{\bt}{\boxtimes}
\newcommand{\ot}{\otimes}

\newcommand{\kk}{\mathbb{k}}

\newcommand{\iso}{\buildrel{\sim}\over{\longrightarrow}}

\begin{document}
\title{The Witt group of non-degenerate braided fusion categories}

\author{Alexei Davydov}
\address{A.D.: Department of Mathematics and Statistics,
University of New Hampshire,  Durham, NH 03824, USA}
\email{alexei1davydov@gmail.com}
\author{Michael M\"uger}
\address{M.M.: Institute of Mathematics, Astrophysics and Particle Physics,
Radboud University, Nijmegen, The Netherlands}
\email{mueger@math.ru.nl}
\author{Dmitri Nikshych}
\address{D.N.: Department of Mathematics and Statistics,
University of New Hampshire,  Durham, NH 03824, USA}
\email{nikshych@math.unh.edu}
\author{Victor Ostrik}
\address{V.O.: Department of Mathematics,
University of Oregon, Eugene, OR 97403, USA}
\email{vostrik@uoregon.edu}

\begin{abstract}
We give a characterization of Drinfeld centers of fusion categories as non-degenerate
braided fusion categories containing a Lagrangian algebra. Further we study the quotient
of the monoid of non-degenerate braided fusion categories modulo the submonoid of the
Drinfeld centers and show that its formal properties are similar to those of the classical
Witt group.
\end{abstract}
\maketitle


\section{Introduction}
Tensor categories are ubiquitous in many areas of mathematics and it seems worthwhile 
to study them deeper. The simplest class of tensor categories is formed by so called
fusion categories (\cite{ENO1}, see \ref{fuscat} below for a definition). It is known (\cite{ENO1}) that 
over an algebraically closed field $\kk$ of characteristic zero there are only countably many 
equivalence classes of fusion categories and that the classification of these equivalence
classes is essentially independent from the field $\kk$ (namely, an embedding of fields 
$\kk \subset \kk'$ induces a bijection between the sets of equivalence classes of fusion categories 
over $\kk$ and over $\kk'$). Thus the classification of fusion categories seems to be 
a natural and interesting problem. This problem is very far from its solution at the moment.

An interesting additional structure that one might impose on a tensor category is a {\em braiding}
(\cite{JS2}). For a fusion category $\A$, its {\em Drinfeld center} $\Z(\A)$ is a braided fusion
category, see Section~\ref{dcenter}. Our first main result addresses the following question: when
is a braided fusion category $\C$ equivalent to the Drinfeld center of some fusion category?
The answer we give is as follows: $\C$ should be {\em non-degenerate} in the sense of \cite{DGNO}
and $\C$ should contain a {\em Lagrangian} algebra, that is, a  connected \'etale algebra of 
maximal possible size, see Section~\ref{qmp}. More precisely, we show that the 2-groupoid of fusion 
categories is equivalent to the 2-groupoid of {\em quantum Manin pairs}, where a quantum Manin
pair consists of a non-degenerate braided fusion category and a Lagrangian algebra in this category.
This result can be considered as (a step in) a reduction of the classification of
all fusion categories to the classification of braided fusion categories. 

The problem of classification of all braided fusion categories (even of non-degene\-rate ones)
seems to be very interesting but is almost as inaccessible as a classification of all fusion categories.
The second main result of this paper is an observation that there is an interesting algebraic
structure in this classification. Namely, we prove that the quotient of the monoid of non-degenerate
braided fusion categories by the submonoid of Drinfeld centers has formal properties similar to those
of the classical Witt group of the quadratic forms over a field. Moreover, we show that the Witt  group
of finite abelian groups endowed with a non-degenerate quadratic forms embeds naturally into
this quotient. Thus we call it  the {\em Witt group of non-degenerate braided
fusion categories} and consider its computation as a fundamental problem in the study of fusion
categories. Further we show that each Witt equivalence class contains a unique representative
which is {\em completely anisotropic} (Theorem~\ref{unica}); this result is a counterpart of the statement 
that in the classical Witt group each Witt class contains a unique anisotropic quadratic form.

An interesting subgroup of the Witt group is the {\em unitary Witt group} (see Definition \ref{Wun})
consisting of the classes of {\em pseudounitary} braided fusion categories.
A well known source of examples of pseudounitary braided fusion categories is the
representation theory of affine Lie algebras, see, e.g., \cite[Chapter 7]{BaKi}. Namely,
for any simple finite dimensional Lie algebra $\fg$ and a positive integer $k$ one constructs
a pseudounitary non-degenerate braided fusion category $\C(\fg,k)$ consisting of integrable
highest weight modules of level $k$ over the affinization of $\fg$.
We do not know any elements of the unitary
Witt group that are not in the subgroup generated by the classes $[\C(\fg,k)]$. It would be
very interesting to find out whether such elements exist.
The relations between the classes $[\C(\fg,k)]$ (or, more generally, between the classes 
of known braided fusion categories) are of great interest. 
By Corollary~\ref{easywitt}, any such relation produces at least one fusion category; one can hope
to construct new examples of fusion categories in this way (see \cite[Appendix]{CMS} for 
an example of this kind). In Section~\ref{vertex} we give examples of such relations using the theory of
{\em conformal embeddings} and {\em coset models} of central charge $c<1$. It would be
interesting to see whether other relations exist. At this moment even all relations between
the classes $[\C(sl(2),k)]$ are not completely known (see Section~\ref{sl2}).

This paper was written under the influence of Vladimir Drinfeld and Alexei Kitaev. 
We are deeply grateful to them for sharing their ideas with us. M.M.\ also thanks A.~Kitaev 
for two invitations to Microsoft's Station Q and to Caltech, respectively. V.O. is grateful 
to Zhenghan Wang for his interest in this work. The work of D.N.\ was partially supported by 
the NSF grant DMS-0800545. The work of V.O.\ was  partially supported by the NSF grant DMS-0602263.
Finally, we would like to thank the referee for his exceptionally thorough work that led to many
clarifications.

\section{Preliminaries}
Throughout this paper our base field $\kk$ is an algebraically closed field of characteristic zero.

\subsection{Fusion categories} 
\label{fuscat}

By definition (see \cite{ENO1}), a {\em multi-fusion category} over $\kk$
is a $\kk-$linear semisimple rigid tensor category with finitely many simple objects and
finite dimensional spaces of morphisms. A multi-fusion category is called a {\em fusion category}
if its unit object $\be$ is simple.  By a {\em fusion subcategory} of a fusion category 
we always mean a full tensor subcategory that is itself fusion (i.e.\ in particular rigid and semisimple.)
Let $\Vec$ denote the fusion category of finite dimensional vector spaces over $\kk$.
Any fusion category $\A$ contains a trivial fusion subcategory consisting of multiples
of $\be$. We will  identify  this subcategory with $\Vec$. A fusion category $\A$
is called {\em simple} if $\Vec$ is the only proper fusion subcategory of $\A$.

A fusion category is called {\em pointed} if all its simple objects are invertible.
For a fusion category $\A$ we denote  $\A_{pt}$ the maximal pointed fusion subcategory
of $\A$. We say that $\A$ is {\em unpointed} if $\A_{pt} =\Vec$.

We will denote $\A \boxtimes \B$ the tensor product of fusion categories $\A$ and $\B$. 
(Cf.\ \cite[Section 5]{De}. Under the assumptions of this paper, where $\kk$ is algebraically closed 
and $\A,\B$ semisimple, $\A\boxtimes\B$ can be obtained 
as the completion of the $\kk$-linear direct product $\A\otimes_\kk\B$ under direct sums and subobjects.)

For a  fusion category $\A$ we denote by $\O(\A)$ the set of
isomorphism classes of simple objects in $\A$.

Let $\A$ be a fusion category and let $K(\A)$ be its Grothendieck ring. 
There exists a unique ring
homomorphism $\FPdim: K(\A)\to \BR$ such that $\FPdim(X)>0$ for any $0\ne X\in \A$, see
\cite[Section 8.1]{ENO1}. (See also \cite[Section 9]{ENO1} for the observation that the results 
used below are independent of the ground field.)
For a fusion category $\A$ one defines (see \cite[Section  8.2]{ENO1}) its
{\em Frobenius-Perron dimension}:
\begin{equation}
\label{FPdim def}
\FPdim(\A)=\sum_{X\in \O(\A)}\,\FPdim(X)^2.
\end{equation}

For any object $X$ in $\A$ let $[X]$ denote the corresponding element of the
Grothendieck ring $K(\A)$.
One defines the (virtual) {\em regular object} of $\A$ by 
\begin{equation}
\label{reg object}
R_\A= \sum_{X\in \O(\A)}\,\FPdim(X)\,[X]\in K(\A)\ot_\mathbb{Z} \BR,
\end{equation}
see, e.g., \cite[Section 8.2]{ENO1}.
The regular object $R_\A$ has the following properties (see {\em loc.\ cit.}):
\begin{enumerate}
\item[(1)] $\FPdim(R_\A)=\FPdim(\A)$.
\item[(2)] $[X]R_\A=\FPdim(X)R_\A$ for any $X\in \A$;
\end{enumerate}
(The first is obvious. The second is a restatement of the fact that the the positive vector 
$(\FPdim(X_i))_i$ is the (unique up to a scalar) common FP eigenvector, with respect to the canonical basis 
$[X_j]$, of the commuting operators $[X]$ acting on $K(\A)\ot_\mathbb{Z} \BR$ by multiplication. The
proof only uses multiplicativity of the FP dimension. This also shows that $R_\A$ is actually
characterized by the properties (1) and (2).)

Let $\A_1,\, \A_2$ be fusion categories such that $\FPdim(\A_1) =\FPdim(\A_2)$.
By \cite[Proposition 2.19]{EO} any fully faithful tensor functor $F: \A_1 \to \A_2$
is an equivalence.

There is another notion of dimension $\A$, the {\em categorical} (or {\em global})
dimension defined as follows (see \cite{Mu-I}).  For each simple object $X$ in $\A$ pick an
isomorphism $a_X: X \xrightarrow{\sim} X^{**}$ and set 
\begin{equation}
\label{catdim def}
\dim(\A)=\sum_{X\in \O(\A)}\,|X|^2,
\end{equation} 
where $|X|^2 = \Tr_X(a_X) \Tr_{X^*}((a_X^{-1})^*)$. By \cite[Theorem 2.3]{ENO1},
$\dim(\A)$ is a non-zero element in $\kk$.

A fusion category $\A$ over $\kk =\mathbb{C}$ is called {\em pseudo-unitary} if $\dim(\A)=\FPdim(\A)$, 
see \cite[Section~ 8.4]{ENO1}. A pseudo-unitary  fusion category $\A$ has a unique spherical structure 
such that the categorical dimension $\dim(X)$ of any object $X$ in $\A$
equals $\FPdim(X)$, see \cite[Proposition 8.23]{ENO1}.  It is easy to see that if $\A_1$ and $\A_2$ are 
pseudo-unitary then so is  $\A_1\bt \A_2$.

\subsection{Braided fusion categories}  

A {\em braided} fusion category  is a fusion category $\C$
endowed with a braiding $c_{X,Y}: X\ot Y\xrightarrow{\sim} Y\ot X$, see \cite{JS2}. 
For a braided fusion category its {\em reverse} $\C^{\rev}$ 
is the same fusion category with a new braiding
$\tilde{c}_{X,Y}=c_{Y,X}^{-1}$. A braided fusion category is {\em symmetric} if
$\tilde{c} = c$.

Recall from \cite{Mu2} that
objects $X$ and $Y$ of a braided fusion category $\C$ are said to
{\em centralize\,}  each other if
\begin{equation} \label{monodromy-drinf}
c_{Y,X}\circ c_{X,Y} =\id_{X\ot Y}.
\end{equation}
The {\em centralizer\,} $\D'$ of a fusion subcategory $\D\subset\C$ is defined to
be the full subcategory of objects of $\C$ that centralize each object of $\D$.
It is easy to see that $\D'$ is a fusion subcategory of $\C$.
Clearly, $\D$ is symmetric if and only if $\D\subset\D'$.

\begin{definition}
\label{nondegdef}
 (see \cite[Definition 2.28 and Proposition 3.7]{DGNO})
We will say that a  braided fusion category $\C$ is {\em non-degenerate}
if $\C'=\Vec$.
\end{definition}

A non-degenerate braided fusion category $\C\neq \Vec$ is {\em prime} if it has no 
proper non-degenerate braided fusion subcategories other than $\Vec$.
Clearly, a non-trivial simple  braided fusion category is prime.

For a fusion subcategory $\D$ of a non-degenerate braided fusion category $\C$ one 
has the following properties, cf.\ \cite[Theorems 3.10, 3.14]{DGNO}:
\begin{gather}
\D''=\D ,\\
\FPdim(\D)\FPdim(\D')=\FPdim(\C).
\end{gather}

A {\em pre-modular} category is a braided fusion category equipped with a spherical structure. 
A pre-modular category $\C$ is {\em modular} (i.e., its $S$-matrix is invertible) if and only
if $\C$ is non-degenerate \cite[Proposition 3.7]{DGNO}. (Cf.\ also \cite{Mu2}.)

The following statement is well known. We include its proof for the reader's convenience.

\begin{proposition}
\label{needed reference}
Let $\C \neq \Vec$ be a non-degenerate braided fusion category. Then
\begin{equation}
\label{tens product}
\C = \C_1\boxtimes \cdots \bt \C_n,
\end{equation}
where $\C_1,\dots, \C_n$  are prime non-degenerate subcategories of $\C$. 
Furthermore, if $\C$ is unpointed  
then  its decomposition \eqref{tens product} into a tensor product of prime non-degenerate subcategories
is unique up to a permutation of factors.
\end{proposition}
\begin{proof}
Existence of the decomposition \eqref{tens product} is established in \cite[Theorems 4.2, 4.5]{Mu2} for 
modular categories. Up to one argument that requires generalization, given by \cite[Theorem 3.13]{DGNO},
the same proof works for non-degenerate fusion categories.

It remains to prove uniqueness.  If $\D\subset \C$ is a fusion subcategory, let $\D_i \subset \C_i$
be the fusion subcategory generated by all simple objects $X_i\in \C_i$ such that 
there is a simple $X =X_1\bt \cdots \bt X_i \bt \cdots \bt X_n \in \D$. Clearly we have
$\D\subset\D_1\boxtimes\cdots\boxtimes\D_n$, but the converse need not hold. If it does, we say that $\D$ 
factorizes. Denoting by
$\D_{ad}$ the fusion subcategory of $\D$ generated by $X\ot X^*$, where $X$ runs through simple objects of $\D$,
the fact that 
$X\ot X^*=(X_1\ot X_1^*)\boxtimes\cdots\boxtimes(X_n\ot X_n^*)$ has
$\be\boxtimes\cdots\boxtimes\be\boxtimes(X_i\otimes X_i^*)\boxtimes\be\boxtimes\cdots\boxtimes\be$
as direct summand for each $i$ implies that $\D_{ad}\supset(\D_{ad})_i$, thus $\D_{ad}$ factorizes. 
Let $\D\subset \C$  be a non-degenerate fusion subcategory. Since $\C$ is unpointed, i.e., $\C_{pt} =\Vec$, $\D$ is
unpointed and by   \cite[Corollary 3.27]{DGNO} we have $\D_{ad} =(\D_{pt})'\cap \D = \D$. Thus $\D$ factorizes, i.e.\
$\D = \D_1\bt \cdots \bt  \D_n$,  where each $\D_i$ is non-degenerate.
Since $\C_i$ is prime, we must have either $\D_i=\Vec$
or $\D_i =\C_i$  for each $i=1,\dots,n$.  In particular, every prime non-degenerate fusion subcategory 
$\D \subset \C$ coincides with some $\C_i$. Hence,  \eqref{tens product} is unique up to a permutation of factors.
%
%
%
\end{proof}

\begin{remark} The proof actually also shows the following stronger result: If $\D\subset\C$ is an unpointed
and non-degenerate fusion subcategory then $\D = \D_1\bt \cdots \bt  \D_n$,  where each $\D_i$ is either 
$\D_i=\Vec$ or $\D_i =\C_i$. This means that the prime factors $\C_i$ that are unpointed appear in every 
prime factorization of $\C$, whether or not $\C$ itself is unpointed.
\end{remark}

\subsection{Drinfeld center of a fusion category} 
\label{dcenter}
For any fusion category $\A$ its {\em Drinfeld center}
$\mathcal{Z}(\A)$ is defined as the category whose objects are
pairs $(X, \gamma_X)$, where $X$ is an object of $\A$ and
$\gamma_X : V\ot X
\simeq X \ot V$, $V\in\A$
is a natural family of isomorphisms, satisfying a certain
compatibility condition, see \cite[Definition 3]{JS1} or
\cite[Defini\-tion~XIII.4.1]{Ka}.
It is known that $\Z(\A)$ is a non-degenerate braided fusion category and that
\begin{equation}
\label{dimZ}
\dim(\Z(\C))=\dim(\C)^2, \quad \FPdim(\Z(\C))=\FPdim(\C)^2.
\end{equation}
(See \cite[Theorems 3.16, 4.14, Proposition 5.10]{Mu4} 
for $\C$ semisimple spherical and 
\cite[Theorem 2.15, Proposition 8.12]{ENO1}, \cite[Corollary 3.9]{DGNO} for $\C$ fusion.)

For a braided fusion category $\C$ there are two braided functors 
\begin{eqnarray}
\C \to \Z(\C) &:& X\mapsto (X,\, c_{-,X}),\label{C_ZC}\\
\C^{\rev} \to \Z(\C) &:& X\mapsto  (X,\, \tilde c_{-,X}).
\end{eqnarray}
These functors are fully faithful and so we can identify
$\C$ and $\C^{\rev}$ with their images in $\Z(\C)$. These images 
centralize each other, i.e., $\C' =\C^{\rev}$. (Cf.\ \cite[Proposition 7.3]{Mu4}.)
This allows to define a braided tensor functor 
\begin{equation}
\label{Gfun}
G: \C \boxtimes \C^{\rev}\to \Z(\C).
\end{equation}
It was shown in \cite[Theorem 7.10]{Mu4} and \cite[Proposition 3.7]{DGNO}
that $G$ is a braided equivalence if and only if $\C$ is
non-degenerate. 

Let $\C$ be a braided fusion category and let $\A$ be a fusion
category. 

\begin{definition}
\label{def centfun} 
If $F: \C \to \A$ is a tensor functor, a structure of a {\em central functor}
on $F$ is a braided tensor functor $F':\C\to\Z(\A )$ whose composition with the forgetful functor
$\Z(\A)\to\A$ equals $F$.
\end{definition}

Equivalently, a structure of central functor on $F$ is a natural family of isomorphisms
$Y\ot F(X) \xrightarrow{\sim} F(X)\ot Y$, $X\in\C$, $Y\in\A$, satisfying certain compatibility conditions,
see \cite[Section  2.1]{Be}.

\subsection{Separable algebras} 

Let $\A$ be a fusion category. In this paper an {\em algebra} $A\in \A$
is an associative algebra with unit, see e.g., \cite[Definition 3.1]{O}.

\begin{definition}
\label{defsep}
An algebra $A\in \A$ is said to be {\em separable} if the
multiplication morphism $m: A\ot A\to A$ splits as a morphism of $A$-bimodules.
\end{definition}

\begin{remark}
\begin{enumerate}
\item[(i)] The morphism $m$ is surjective (due to the existence of unit in $A$), so the definition
makes sense.
\item[(ii)] Observe that if $F:\A \to \B$ is a tensor functor then $F(A)\in \B$ is a separable algebra for
a separable algebra $A\in \A$.
\end{enumerate}
\end{remark}

For an algebra $A\in \A$ let $\A_A$, ${}_A\A$, ${}_A\A_A$ denote, respectively, abelian categories
of right $A-$modules, left $A-$modules, $A-$bimodules, see e.g., \cite[Definition 3.1]{O}.

\begin{proposition} 
\label{semis eq}
For an algebra $A\in \A$ the following conditions are equivalent:
\begin{enumerate}
\item[(i)]  $A$ is separable;
\item[(ii)]  the category $\A_A$ is semisimple;
\item[(iii)] the category ${}_A\A$ is semisimple;
\item[(iv)] the category ${}_A\A_A$ is semisimple.
\end{enumerate}
\end{proposition}
\begin{proof} Assume that $A$ is separable. Note that $A$ considered as a  bimodule over itself 
is a direct summand of the $A-$bimodule $A\ot A$. Thus any $M=M\ot_AA\in \A_A$ is a direct summand of $M\ot_AA\ot A=M\ot A$.
The object $M\ot A\in \A_A$ is projective (see e.g. \cite[Section  3.1]{O}). Thus any $M\in \A_A$ is
projective and we have implication (i)$\Rightarrow$(ii). The implication (i)$\Rightarrow$(iii) is
proved similarly.

The implications (ii)$\Rightarrow$(iv) and (iii)$\Rightarrow$(iv) follow from \cite[Theorem 2.16]{ENO1}
and \cite[Remark 4.2]{O}. Finally, the implication (iv)$\Rightarrow$(i) is obvious.
\end{proof}

Let $\C$ be a braided fusion category. Recall that an algebra $A$ in $\C$ is called
{\em commutative} if $m \circ c_{A,A} =m$, where $m: A\ot A \to A$ is the multiplication
of $A$, see e.g., \cite[Definition 1.1]{KiO}.

\begin{example}
\label{regular}
Let $G$ be a finite group and let $\A =\Rep(G)$ be the fusion category of finite
dimensional representations of $G$. Let $A=\Fun(G)$ be the algebra of $\kk-$valued functions on $G$.
The group $G$ acts on $A$ via left translations, so $A$ can be considered as a
commutative  algebra in $\A$.
The algebra $A$ is called the {\em regular} algebra of the category $\A=\Rep(G)$. 
Associating to $f\in A$ the function $\mu(f): G\times G\rightarrow\kk, (g,h)\mapsto \delta_{g,h}f(g)$,
easy computations show that $\mu:A\rightarrow A\otimes A$ is a splitting of 
$m:A\otimes A\rightarrow A$ and a bimodule map. Thus $A$ is separable. (Cf.\ \cite[p.\ 227]{Br} for
a similar argument.)

More generally we say that a braided fusion category $\E$ is {\em Tannakian} \cite{De} if there is a
braided equivalence $F: \E \simeq \Rep(G)$; in this case the algebra $F^{-1}(A)$ (with $A\in \Rep(G)$
as above) is called a regular algebra $A_\E$ of $\E$. It is known that the algebra $A_\E$ is unique
up to isomorphism. (Such an isomorphism is non-unique, in particular $\Aut\,A_\E\cong G$.) 
See, e.g., \cite[Section 2.13]{DGNO}. 
\end{example}

\subsection{Equivariantization and de-equivariantization}
\label{eq and deeq}

Let $\A$ be a fusion category with an action of a finite 
group $G$. In this case one can define the fusion category $\A^G$ of 
$G$-equivariant objects in $\A$. 
Objects of this category are objects $X$ of $\A$ equipped with 
an isomorphism $u_g: g(X)\to X$ for all $g\in G$, such that
$$
u_{gh}\circ \gamma_{g,h}=u_g\circ g(u_h),
$$
where $\gamma_{g,h}: g(h(X))\to gh(X)$ is the natural isomorphism
associated to the action. Morphisms and tensor
product of equivariant objects 
are defined in an obvious way. This category is called
the {\em $G$-equivariantization} of $\A$. One has 
$\FPdim(\A^G)=|G|\FPdim(\A)$.  See \cite{Br, Mu3} and \cite[Section 4]{DGNO}
for details.

\begin{example} 
\label{GHequiv}
Let $H$ be a normal subgroup of $G$. Then there is a natural action of $G/H$ on
$\A^H$ and $(\A^H)^{G/H}\cong \A^G$.
\end{example}

There is a procedure opposite to equivariantization, called the 
{\em de-equivariantiza\-tion}.
Namely, let $\A$ be a fusion category  and let
$\E= \Rep(G) \subset \Z(\A)$  be a Tannakian subcategory which embeds into $\A$
via the forgetful functor $\Z(\A)\to \A$. Let $A=\mbox{Fun}(G)$ be the regular algebra of $\E$.
It is a separable commutative algebra in $\Z(\A)$ and so
the category $\A_G$ of left $A$-modules in $\A$ is a
fusion category  with the tensor product $\ot_A$, 
called {\em de-equivariantization} of $\A$. 
One has $\FPdim(\A_G) = \FPdim(\A)/ |G|$.

The above constructions are canonically inverse to each other, i.e., there are
canonical equivalences $(\A_G)^G \cong \A$ and $(\A^G)_G \cong \A$,
see \cite[Section 4.2]{DGNO}.

\subsection{Module categories over fusion categories} \label{dualmcat}

Let $\A$ be a fusion category.
A left {\em $\A$-module category} is a finite semisimple Abelian  $\kk$-linear
category $\M$ together with  a bifunctor $\ot: \A \times \M \to \M$ and a natural family
of isomorphisms 
\[
(X\ot Y) \ot M  \xrightarrow{\sim}  X\ot (Y \ot M)
\quad  \mbox{and} \quad
\be\ot M  \xrightarrow{\sim}  M
\]
for $X,\,Y\in \A, \, M\in \M$, satisfying certain coherence
conditions. See \cite{O} for details and for the definitions of
$\A$-module functors and their natural transformations.
A typical example of a left $\A$-module category is the category $\A_A$
of  right modules over a separable algebra $A$ in $\A$ \cite{O}.
An $\A$-module category is called {\em indecomposable} if it is not
equivalent to a direct sum of two non-trivial $\A$-module categories.

The category of $\A$-module endofunctors of a right $\A$-module category
$\M$ will be denoted by $\A^*_\M$.  It is known that $\A^*_\M$
is a multi-fusion category, see \cite[Theorem 2.18]{ENO1} (it
is a fusion category if and only if $\M$ is indecomposable).

Let $\M$ be an indecomposable right $\A$-module category.
We can regard $\M$ as an $(\A^*_\M,\A)$-bimodule category.
Its $(\A^*_\M, \A)$-bimodule endofunctors can be identified, on the one hand,
with functors of left multiplication by objects of  $\Z(\A_\M^*)$, and on the other hand,
with functors of right multiplication by objects of $\Z(\A)$. Combined, these identifications
yield a canonical equivalence of braided categories
\begin{equation}
\label{Schauenburg}
\Z(\A)\xrightarrow{\sim} \Z(\A_\M^*).
\end{equation}
This result is due to Schauenburg, see  \cite{Sch}.

\section{\'Etale algebras and central functors}

\subsection{\'Etale algebras in braided fusion categories}

\begin{definition} 
\label{etale}
An algebra $A\in \C$ is said to be {\em \'etale} if
it is both commutative and separable.
We say that an \'etale algebra $A\in \C$ is {\em connected} if $\dim_\kk \Hom_\C(\be,\, A)=1$.
\end{definition}

\begin{remark} \label{disconn}
\begin{enumerate}
\item[(i)] The terminology of Definition~\ref{etale} is justified by the fact that  \'etale algebras 
in the usual sense can be characterized by the property from Definition~\ref{etale}.
\item[(ii)] Any \'etale algebra canonically decomposes as a direct sum of connected ones.
\end{enumerate}
\end{remark}

\begin{example}\label{regularce}
\begin{enumerate}
\item[(i)]  Let $\E \subset \C$ be a Tannakian subcategory. Then a regular algebra
$A_\E\in \C$ (see Example \ref{regular}) is connected \'etale.
\item[(ii)] Let $\C$ be a pre-modular category. Let $A$ be a commutative algebra  in
$\C$ such that $\dim_\kk \Hom_\C(\be ,A)=1$, the pairing $A\otimes A\xrightarrow{m}
A\twoheadrightarrow\be$ is non-degenerate, $\theta_A=\id_A$, and $\dim(A)\ne 0$. It is proved
in \cite[Theorem 3.3]{KiO} that such an $A$ is connected \'etale.
\end{enumerate}
\end{example}

\begin{remark} \label{etsd}
In general if $A\in \C$ is a connected \'etale algebra and $A\twoheadrightarrow\be$ is 
a nonzero homomorphism (it is unique up to a scalar) then the pairing $A\ot A\xrightarrow{m}
A\twoheadrightarrow\be$ is non-degenerate. Indeed the kernel of this pairing would be a non-trivial
ideal of $A$ (= non-trivial subobject in the category $\C_A$); but the category $\C_A$ is
semisimple and $\dim_\kk \Hom_{\C_A}(A,\,A)=\dim_\kk \Hom_\C(\be,\,A)=1$. In particular, this implies
that any \'etale algebra is a self-dual object of $\C$ (use Remark \ref{disconn} (ii) for disconnected
\'etale algebras).
\end{remark}

\subsection{From central functors to \'etale algebras}
\label{catCA1}

\begin{lemma}
\label{centraletale}
Let $\C$ be a braided fusion category, let $\A$ a fusion category, and let $F: \C \to \A$
be a central functor. Let $I : \A \to \C$ be the right adjoint functor of $F$. Then the object
$A=I(\be)\in \C$ has a canonical structure of connected \'etale algebra.

The category of right $A$-modules in $\C$ is monoidally equivalent to the image of $F$, i.e.\ the
smallest fusion subcategory of $\A$ containing $F(\C)$.
\end{lemma}
\begin{proof}
Let $\phi: \C \to \Vec$ be the contravariant representable functor corresponding to $A$,
that is, $\phi(X)=\Hom_\C(X,\,A) \cong \Hom_\A(F(X),\,\be)$. The linear map
\begin{multline*}
\Hom_\A(F(X_1),\,\be)\ot_\kk \Hom_\A(F(X_2),\,\be)\to  \\
\Hom_\A(F(X_1)\ot F(X_2),\,\be \ot \be) \cong  \Hom_\A(F(X_1\ot X_2),\be)
\end{multline*}
defines a natural family
\begin{equation}
\label{nu}
\nu_{X_1,X_2}: \phi(X_1)\ot_\kk \phi(X_2)\to \phi(X_1\ot X_2)
\end{equation}
such that the compositions
\begin{equation}\begin{array}{ccccl}
\phi(X_1)\ot \phi(X_2)\ot \phi(X_3) &\to& \phi(X_1\ot X_2)\ot \phi(X_3)&\to& \phi(X_1\ot X_2\ot X_3), \\
\phi(X_1)\ot \phi(X_2)\ot \phi(X_3) &\to& \phi(X_1)\ot \phi(X_2\ot X_3)&\to& \phi(X_1\ot X_2\ot X_3)
\end{array}\label{assoc}  \end{equation}
are equal. We claim that a morphism \eqref{nu} is the same thing as an associative multiplication 
$m: A\ot A\to A$. Namely, we define $m\in\Hom(A\ot A,A)=\phi(A\ot A)$ by $m:=\nu_{A,A}(\id_A\ot \id_A)$,
where $\id_A$ is considered as an element of $\phi(A)$. Now by naturality of $\nu$ one has
\[ \nu_{X_1,X_2}(f\otimes g)=m\circ(f\otimes g), \]
and associativity of $m$ follows from (\ref{assoc}).

By definition, $\Hom_\C(\be,\,A)=\Hom_\A(F(\be),\be)=\Hom_\A(\be,\be)=\kk$. It is
easy to see that the image of $1\in \kk$ in $\Hom_\C(\be,A)$ is a unit of the algebra $A$.

Next we want to prove the commutativity of $m$. By its definition, $m$
is the image of a certain morphism $\tilde{m}\in\Hom_\A(F(A\ot A),\be)$
under the bijection 
\[
\Hom_\A(F(A\ot A),\be)\cong\Hom_\C(A\ot A,A). 
\]
By naturality of
the adjunction bijections, $m\circ c_{A,A}$ corresponds to 
$\tilde{m}\circ F(c_{A,A})\in\Hom_\A(F(A\ot A),\be)$. The equality
$\tilde{m}=\tilde{m}\circ F(c_{A,A})$ follows from commutativity of
the following diagram, where $F'$ is the central structure,
i.e.\ a braided tensor functor $F':\C\rightarrow \Z(\A)$ lifting $F:\C\rightarrow\A$. 

\begin{equation*}
\xymatrix{F'(A\ot A)\ar[rr]^{\sim} \ar[d]_{F'(c_{A,A})}
&& F'(A)\ot F'(A)\ar[rr]^{\quad \quad l\ot l} \ar[d]_{c_{F'(A),F'(A)}}
&& \be\ot\be \ar[d]^{c_{\be,\be}} \ar[rr]^{\sim} &&\be \ar[d]_{\id_\be}\\
F'(A\ot A) \ar[rr]^{\sim} && F'(A)\ot F'(A)\ar[rr]^{\quad \quad l\ot l}&& \be\ot\be\ar[rr]^{\sim} &&\be}
\end{equation*}
Here 
$l\in\Hom_\C(F(A),\be)$ is the
image of $\id_A$ under $\Hom_\C(A,A)\cong \Hom_\A(F(A),\be)$. The left
square commutes since $F'$ is a braided functor, and the right one
since $c_{\be,\be}=\id$. That the middle square commutes is more
subtle, since $l:F(A)\rightarrow\be$ only is a morphism in $\A$, but not in $\Z(\A)$.
It commutes nevertheless since the braiding of $\Z(\A)$ is natural for
such morphisms w.r.t.\ the second argument. (Since
$c_{(X,e_X),(Y,e_Y)}=e_X(Y)$ and the half-braiding $Y\mapsto e_X(Y)$ is natural w.r.t.\ all
morphisms $Y\rightarrow Y'$ in $\A$.)

That the category of right $A$-modules in $\C$ identifies with the image of $F$ in $\A$
follows from \cite[Theorem 3.17]{EO} (cf.\ also \cite[Theorem 3.1]{O}).
Thus $\C_A$ is semisimple. By Proposition~\ref{semis eq}
semisimplicity of the category of $A$-modules implies the semisimplicity of
the category of $A$-bimodules. In particular, the morphism of $A-$bimodules $m: A\ot A\to A$, thus
$A$ is separable.
\end{proof}

\begin{example}
\begin{enumerate}
\item[(i)] Let $\C=\Rep (G)$ and $F:\C\to\Vec$ the forgetful functor. Then the \'etale algebra $A$
from Lemma \ref{centraletale} is the regular algebra, see Example~\ref{regular}.
\item[(ii)]  Let $\Vec_G^\omega$ be the fusion category of finite dimensional $G$-graded
vector spaces with the associativity constraint twisted by a $3$-cocycle
$\omega\in Z^3(G,\, \kk^\times)$.
Let $\C =\Z(\Vec_G^\omega)$ and $F: \C \to \Vec_G^\omega$ the forgetful functor.
Then the  \'etale algebra $A$ from Lemma~\ref{centraletale} is the regular algebra of
$\Rep(G)\subset \C$.
\item[(iii)] Let $\C =\Z(\Rep(G))\cong \Z(\Vec_G)$ and $F: \C \to \Rep(G)$ the forgetful functor.
Then the \'etale algebra $A$ from Lemma \ref{centraletale} is the group algebra of $G$ considered
as an algebra in $\C$. Notice that in this case the algebra $F(A)$ in the symmetric
tensor category $\Rep(G)$ is non-commutative unless $G$ is commutative.
\end{enumerate}
\end{example}

\begin{remark} Lemma \ref{centraletale} fails over fields of characteristic $p>0$.
Namely the algebra $A=I(\be)$ is still commutative (with the same proof) but it can
fail to be separable. Here is a
counter-example. 
Let $G$ be a finite abelian group of order divisible by $p$. Take $\C=\Vec_G$,
i.e., $\C$ is the category of finite-dimensional $G$-graded vector spaces
with the obvious symmetric braided structure. Let $\D=\Vec$ and let $F:\C \to \D$ be the
functor of forgetting the grading. Then $A$ is the group algebra of $G$, which is not
\'etale. In this example the category of $A$-bimodules identifies with $\Rep(G)$ and is not semisimple.
\end{remark}

\subsection{The tensor category $\C_A$ corresponding to an \'etale algebra $A$}
\label{catCA2}

Let $\C$ be a braided fusion category and let $A\in \C$ be a connected \'etale algebra. 
Let $\C_A$ be the category of right $A$-modules and let 
\begin{equation}
\label{functor FA}
F_A: \C \to \C_A : X \mapsto X\ot A
\end{equation}
be the free module functor.
The category $\C_A$ is semisimple by Proposition~\ref{semis eq}.

Using the braiding we can define two left $A$-module structures on a right $A$-module $M$ by
\begin{equation}
\label{pm}
 A\otimes M\xrightarrow{c_{A,M}}M\otimes
A\to M \quad \mbox{ or by } \quad  A\otimes M\xrightarrow{c_{M,A}^{-1}}M\otimes
A\to M.
\end{equation}
Both structures make $M$ 
an $A$-bimodule, and we will denote the results by $M_+$ and $M_-$, respectively.
Clearly, the functors $M\mapsto M_\pm$  are sections
of the forgetful functor ${}_A \C_A\to \C_A$.

Since the category  ${}_A \C_A$ of $A$-bimodules in $\C$ is a tensor category, we obtain in this way 
two tensor structures $\ot_\pm$  on $\C_A$ which are opposite to each other. For definiteness,
when we consider $\C_A$ as a tensor category, the tensor structure $\ot_-$ is understood. By definition, we have 
tensor functors $\C_A\to{}_A\C_A$ and $\C_A^{\rev}\to{}_A\C_A$. 

Now the functor $F_A:\C\to\C_A$ has an obvious structure of tensor functor. The category
$\C_A$ is rigid since any object $M$ in $\C_A$ is a direct summand of the rigid object
$F_A(M)=M\ot A=M\ot_A(A\ot A)$. The unit object of $\C_A$ is $A=F_A(\be)$ and the connectedness
of $A$ implies that $A\in \C_A$ is simple. Thus, $\C_A$ is a fusion category. Alternatively, this
follows from the fact that ${}_A\C_A$ is fusion, cf.\ e.g.\ \cite{O}, and the fact that the functors
$M\mapsto M_\pm$ from $\C_A$ and $\C_A^{\rev}$ to ${}_A \C_A$ are tensor embeddings.


\begin{example}
\label{regulardeeq}
 Let $\C$ be a braided fusion category and let $\E \subset \C$ be a Tannakian
subcategory. Let $A\in \E$ be the regular algebra (which is connected \'etale by
Example \ref{regularce} (i)). In the terminology of  \cite[Section 4.2]{DGNO}
the fusion category $\C_A$ introduced above is 
the de-equivariantization of $\C$ (cf.\ Section~\ref{eq and deeq}) viewed as a fusion category over~$\E$.
\end{example}

\subsection{The central functor $\C\to\C_A$}\label{eacf}
Observe that the free module functor \eqref{functor FA}
admits a natural structure of a central functor, see Definition~\ref{def centfun}. 
Indeed, we have $F_A(X)=X\otimes A$, and, hence,
$F_A(X)\otimes_A Y = X\otimes Y$. Similarly, $Y\otimes_A
F_A(X)=Y\otimes X$. These two objects are isomorphic via the
braiding of $\C$ (using the commutativity of $A$, one can  check that the braiding gives an
isomorphism of $A$-modules) and,
hence, $F_A$ lifts to a braided tensor functor
\begin{equation}
\label{tilde functor FA}
 F'_A: \C \to \Z(\C_A)
\end{equation}
whose composition with the forgetful functor $\Z(\C_A) \to \C_A$ equals $F_A$.
This construction is in  a sense converse to Lemma~\ref{centraletale}:

\begin{lemma} \label{eacfl}
Let $A\in\C$ be a connected \'etale algebra and let $F_A: \C\rightarrow\C_\A$ 
be the central functor as above. Then the algebra object $A_{F_A}=I(\be)$ obtained from $F_A$ 
according to Lemma \ref{centraletale} is isomorphic to $A$.
\end{lemma}

\begin{proof} 
The adjoint of the functor 
$F_A:\C\rightarrow\C_A$ is 
  given by the forgetful functor $I:\C_A\rightarrow\C$. The unit of
  $\C_A$ being $(A,m)$, we have $I(\be_{\C_A})=A$. It is
  straightforward to see that the construction of the algebra
  structure on $A=I(\be_{\C_A})$ defined in (the proof of) Lemma
  \ref{centraletale} recovers the original algebra structure.
\end{proof}


Let $\A_1,\, \A_2$ be fusion categories.
We will say that a tensor functor $F:\A_1 \to \A_2$ is {\em surjective} if any object in $\A_2$ is
a subobject of some $F(X),\ X\in\A_1$. 

\begin{remark}
Some authors use the term {\em dominant} functor for what we call a surjective functor,
see \cite{Br, BrN}. 
\end{remark}

\begin{lemma} 
\label{dim CAlemma}
For a connected \'etale algebra $A$ in a braided fusion category $\C$ we have
{\em
\begin{equation}
\label{FPdimCA}
\FPdim(\C_A)=\frac{\FPdim(\C)}{\FPdim(A)}.
\end{equation}}
\end{lemma}
\begin{proof}
The functor \eqref{functor FA} is surjective. Considering the multiplicity of the unit object
on both sides of the identity proven in \cite[Proposition 8.8]{ENO1}, we obtain
\[
\frac{\FPdim(\C)}{\FPdim(\C_A)} = \sum_{X\in \O(\C)}\, \FPdim(X) [F_A(X):\be] =\FPdim(I(\be)),
\]
where $\O(\C)$ denotes the set of simple objects of $\C$ and $I$ is the right adjoint of $F_A$.
Since $A =I(\be)$, the result follows.
\end{proof}

\subsection{Subcategory $\C_A^0\subset \C_A$ of dyslectic modules}
Let $\C$ be a braided fusion category and $A\in \C$ be a connected \'etale
algebra and recall the discussion of the tensor functors 
$M\mapsto M_\pm$ from $\C_A$ and $\C_A^{\rev}$ to ${}_A \C_A$ in Subsection \ref{catCA2}.


\begin{definition}
\label{def dysl}
A module $M\in\C_A$ is
{\em dyslectic} (or {\em local}, in alternative terminology)
if 
the identity map $\id_M$ is an isomorphism of $A$-bimodules $M_+\simeq M_-$.
\end{definition}

Equivalently, a module $M\in \C_A$ is dyslectic
if the following diagram
\begin{equation}
\xymatrix{
M \ot A \ar[rr]^{c_{A,M}\circ c_{M,A}} \ar[dr]_{\rho} &&
M\ot A \ar[dl]^{\rho} \\
& M &
}
\end{equation}
commutes. Here  $\rho: M\ot A \to M$ denotes the action of $A$ on $M$.

The notion of dyslectic module was introduced by Pareigis in \cite{P}. See also \cite{KiO}.

\begin{remark}
Note that a simple $M \in \C_A$  is dyslectic if and only if $M_+ \simeq M_-$ as $A$-bimodules.
Indeed,  since the functors $M \mapsto M_\pm$ from $\C_A$ to
${}_A \C_A$ are embeddings, for
any simple $M \in \C_A$ any isomorphism between $A$-bimodules $M_+$ and $M_-$
must be a multiple of $\id_M$.
\end{remark}

Dyslectic modules form a  full subcategory of $\C_A$ which will be denoted
by $\C_A^0$.
It is known (see \cite[Section 2]{P} and \cite{KiO}) that $\C_A^0$ is closed
under $\ot_A$ and  that the braiding in $\C$ induces a natural braided structure
in $\C_A^0$.  Thus, $\C_A^0$ is a braided fusion category.


\begin{example} \label{tannaka rep0}
Let $\E \subset \C$ be a Tannakian subcategory and let $A\in \E$
be a regular algebra, see Example~\ref{regular}. Then \cite[Proposition 4.56(i)]{DGNO}
says that $\C_A^0$ is equivalent to the de-equivariantization
of $\E'$, cf.\ Section~\ref{eq and deeq}.
\end{example}

\begin{lemma}
\label{when XA is dys}
Let $\C$ be a braided fusion category, let $A$ be an \'etale
algebra in $\C$, and let $X$ be an object of $\C$. Then
the free module $X\ot A$ is dyslectic if and only if $X$ centralizes $A$.
\end{lemma}
\begin{proof}
Consider the following diagram, where we omit identity maps and associativity constraints:
\begin{equation}
\xymatrix{
&&  A \ot X \ot A   \ar[dr]_{c_{A,X}} \ar[drr]^{c_{A,X\ot A}} &&  \\
X\ot A\ot A \ar[r]_{c_{A,A}} \ar[urr]^{c_{X\ot A,A}} \ar[drr]_{m_A} &
X\ot A\ot A \ar[ur]_{c_{X,A}} \ar[dr]^{m_A}  &&
X\ot A\ot A \ar[r]_{c_{A,A}} \ar[dl]_{m_A} & X\ot A\ot A \ar[dll]^{m_A} \\
&& X\ot A. &&
}
\end{equation}
The two upper triangles commute by the hexagon axioms and the two lower
triangles commute since $A$ is commutative.  Therefore,
\[
(\id_X \ot  m_A) \circ (c_{A,X}\circ c_{X,A} \ot \id_A)=
(\id_X \ot  m_A) \circ c_{A, X\ot A} \circ c_{X\ot A,A} \circ (\id_X \ot c_{A,A}^{-1}),
\]
which means that $X\ot A$ is dyslectic if and only if
\begin{equation}
\label{mid diamond}
(\id_X \ot  m_A) \circ (c_{A,X}\circ c_{X,A} \ot \id_A) = \id_X \ot  m_A.
\end{equation}
In other words, commutativity
of the perimeter of the above  diagram is equivalent to commutativity
of the diamond in the middle.   Let $u_A: \be \to A$ denote the unit of $A$.
Suppose that \eqref{mid diamond} holds.  We have
\begin{eqnarray*}
c_{A,X}\circ c_{X,A}
&=&  (\id_X \ot  m_A) \circ (\id_{X\ot A}\ot u_A) \circ  c_{A,X}\circ c_{X,A}  \\
&=&  (\id_X \ot  m_A) \circ (c_{A,X}\circ c_{X,A} \ot \id_A) \circ (\id_{X\ot A}\ot u_A) \\
&=& (\id_X \ot  m_A) \circ  (\id_{X\ot A}\ot u_A)  = \id_{X\ot A}.
\end{eqnarray*}
where the third equality holds by \eqref{mid diamond}.
Thus, \eqref{mid diamond} is equivalent to
$c_{A,X}\circ c_{X,A}=\id_{X\ot A}$. Combining the above equivalences we get the result.
\end{proof}

\subsection{\'Etale algebras in $\C_A^0$ and \'etale algebras over $A$}
Let $\C$ be a braided fusion category and let $A\in \C$ be a connected \'etale algebra. 
An algebra $B\in \C$ equipped with a unital homomorphism $f: A\to B$ is called {\em algebra
over $A$} if the following diagram commutes:
\begin{equation*}
\xymatrix{
A \ot B \ar[r]^{f\ot \id}\ar[d]_{c_{B,A}^{-1}}&B\ot B\ar[r]^{m_B}&B\\
B\ot A \ar[r]^{\id \ot f}
&B\ot B\ar[ur]_{m_B} &
}
\end{equation*}
In the language of \cite[\S 5.4]{O} we require that the morphism $f$ lands in the right center of $B$;
in particular for a commutative algebra $B$ this diagram commutes automatically. Notice that
the morphism $f$ is automatically injective since the algebra $A$ has no nontrivial right ideals.

Observe that an algebra $B$ over $A$ has an obvious structure of right $A-$module, that is $B\in \C_A$.
Moreover, any right $B-$module has an obvious structure of right $A-$module.
The following statements are tautological:

(a) An algebra over $A$ is the same as an algebra in $\C_A$.

(b) Let $B$ be an algebra over $A$. Then right $B-$module in $\C_A$ 
(with $B$ considered as an algebra
in $\C_A$) is the same as right $B-$module in $\C$. In particular, the categories 
$(\C_A)_B$ and $\C_B$ are equivalent. 

(c) Commutative algebra over $A$ is the same as commutative algebra in $\C_A^0\subset \C_A$.


\begin{proposition} \label{overetale}
{\em (cf. \cite[Lemma 4.13]{FFRS} and \cite[Proposition 2.3.3]{Da1})}
A commutative algebra over $A$ is \'etale if and only if the corresponding algebra in $\C_A^0$ is
\'etale. Under this bijection connected algebras correspond to connected ones.
\end{proposition}

\begin{proof} The first statement follows from the tautologies above combined with 
Proposition \ref{semis eq}. The second statement is implied by the fact that a simple $A-$module $M$
with $\Hom_\C(\be,M)\ne 0$ is isomorphic to $A$, see e.g. \cite[Lemma 3.2]{O}.
\end{proof}

\subsection{The category $\Rep_\A(A)$ and its center}
\label{schauenburg}
Let $\A$ be a fusion category and let $F: \Z(\A)\to \A$ be the forgetful functor.
Let $A\in \Z(\A)$ be a connected \'etale algebra. Observe that any right $F(A)$-module
$M\in \A$ has a natural structure of left $F(A)$-module defined as $F(A)\ot M\xrightarrow{\sim} 
M\ot F(A)\to M$.
It is easy to verify that in this way $M$ acquires a structure of $F(A)$-bimodule.

\begin{definition} The category $\Rep_\A(A)$ is the tensor category of right $F(A)$-modules
in $\A$ with tensor product $\ot_{F(A)}$.
\end{definition}

\begin{remark} \label{left=right}
\begin{enumerate}
\item[(i)] Assume that $\C$ is a braided fusion category and $A\in \C$ is
a connected \'etale algebra. Then $A$ can be considered as a connected \'etale
algebra in $\Z(\C)$ via the braided functor $\C \to \Z(\C)$ given in (\ref{C_ZC}). 
In this case the categories
$\C_A$ and $\Rep_\C(A)$ are identical. Nevertheless the tensor structures on
$\C_A$ and $\Rep_\C(A)$ are {\em opposite} to each other.
\item[(ii)] The category $\Rep_\C(A)$ is equivalent to the category of {\em left} $F(A)-$modules.
\end{enumerate}
\end{remark}

Arguments similar to those in Section~ \ref{catCA2} show that $\Rep_\A(A)$ is a semisimple
rigid tensor category. Its unit object $F(A)$ may be reducible, so in general $\Rep_\A(A)$ is not
a fusion category. In general $\Rep_\A(A)$ is an example of a multi-fusion category,
see Section~\ref{fuscat}.

\begin{remark} \label{brn}
Given an \'etale algebra $A\in \Z(\A)$ there is a surjective tensor functor 
\[
\A \to \Rep_\A(A) : X\mapsto X\ot F(A).
\]
Conversely, let $G: \A \to \B$ be a tensor functor and let $I: \B \to \A$ be its right adjoint.
Then the object $I(\be)\in \A$ has a natural lift to $\Z(\A)$. Moreover, it has a natural structure
of an \'etale algebra in $\Z(\A)$. The algebra $I(\be)\in \Z(\A)$ is connected if and only if
the functor $G$ is not decomposable into a non-trivial direct sum of tensor functors.
Similarly to Section~ \ref{eacf} these two constructions are inverse to each other.
See \cite{BrN} for details.
\end{remark}

It is easy to see that the forgetful functor $\Z(\A)_A^0\hookrightarrow \Z(\A)_A\to \Rep_\A(A)$ has a
canonical structure of central functor. Thus, it lifts to a braided tensor functor
\begin{equation}
\label{schfun}
\Z(\A)_A^0\to \Z(\Rep_\A(A)). 
\end{equation}
The following result was proved by Schauenburg (see \cite[Corollary 4.5]{Sch})
under much weaker assumptions on the category $\A$ and commutative algebra
$A\in \Z(\A)$ than ours.

\begin{theorem}\label{petersch}
The functor \eqref{schfun} is a braided equivalence $\Z(\A)_A^0\cong \Z(\Rep_\A(A))$.
\end{theorem}

\noindent
{\em Sketch of proof.} We just sketch a construction of an inverse functor. Let $M\in \Z(\Rep_\A(A))$.
For any $X\in \A$ consider $X\ot F(A)\in \Rep_\A(A)$. Then 
\[
(X\ot F(A))\ot_{F(A)}M=X\ot M \quad \mbox{and} \quad M\ot_{F(A)}(X\ot F(A))=M\ot X. 
\]
It is easy to see now that the central structure of $M$ as
$F(A)$-module defines a central structure of $M$ as an object of $\A$. Moreover one
verifies directly that $F(A)$-module structure on $M$ gives $A$-module structure on
this lift of $M$ to $\Z(\A)$; the resulting object of $\Z(\A)_A$ lies in $\Z(\A)_A^0$. Finally, this
assignment has a natural structure of tensor functor. \hfill \qedsymbol

\begin{remark} \label{indmult}
Theorem \ref{petersch} implies that the unit object of the fusion category $\Z(\Rep_\A(A))$
is indecomposable (recall that the algebra $A$ is connected). It follows that the multi-fusion category
$\Rep_\A(A)$ is {\em indecomposable} in the sense of \cite[Section 2.4]{ENO1}. 
\end{remark}

\subsection{Properties of braided tensor functors}

\begin{proposition}
\label{AinC'}
Let $\C,\,\D$ be braided fusion categories and let $F:\C \to \D$ be a surjective
braided tensor functor.  Let $I:\D \to \C$  be the right adjoint functor of $F$ and
let $A:=I(\be)$ be the canonical connected  \'etale algebra constructed in Lemma~\ref{centraletale}.
Then $A\in \C'$.
\end{proposition}
\begin{proof}
Since $F$ is a central functor,  $\D$ identifies with the category $\C_A$
of $A$-modules in $\C$, cf.\ Section~\ref{eacf}.  We claim that every $A$-module is
dyslectic, i.e., that $\C_A =\C_A^0$.  Indeed,  the fusion category ${}_A \C_A$
identifies with the category of $\C$-module endofunctors of $\D$, see \cite{O}
(the action of $\C$ on $\D$ is defined via $F:\C \to \D$).
Under this identification, for every simple object $M\in \D$ the bimodules $M_\pm$ correspond
to endofunctors of left and right multiplication by $M$.  But these endofunctors
are isomorphic via the braiding of $\D$, i.e., $M$ is dyslectic.

In particular, for every $X\in \C$ the free $A$-module $X\ot A$ is dyslectic.  Hence,
Lemma~\ref{when XA is dys} implies that every $X\in \C$ centralizes $A$, i.e.,
$A\in \C'$.
\end{proof}

\begin{remark}
\label{A=FunG/H}
Note that the \'etale algebra $A$ from Proposition~\ref{AinC'} is a commutative algebra
in a symmetric fusion category $\C'$. Therefore, $A$ belongs to
the maximal Tannakian  subcategory $\E=\Rep(G) \subset \C'$. As is well known, every 
\'etale algebra $A\in\Rep(G)$ is isomorphic to the algebra $\Fun(G/H)$ of functions on $G$ 
invariant under translations by elements of $H$ for a uniquely determined subgroup 
$H\subset G$, the module category $\Rep(G)_A$ is equivalent to $\Rep(H)$ and the functor 
$F_A$ identifies with the restriction functor $\Rep(G)\to \Rep(H)$. In view of $A\in\E$, 
the restriction $F:\E \to F(\E)$ of $F$ to $\E$ identifies with the restriction functor
$\Rep(G)\to \Rep(H)$.
\end{remark}

\begin{corollary}
Let $F: \C_1\to \C_2$ be a surjective braided tensor functor between
braided fusion categories. There exists a braided fusion category $\C$
with an action of a finite group $G$, a subgroup $H\subset G$, and
braided tensor equivalences $\C_1\cong~\C^G,\, \C_2\cong~\C^H$
such that the diagram
\begin{equation}
\xymatrix{
\C_1 \ar[rr]^{F} \ar[d]^{\wr}  &&
\C_2 \ar[d]^{\wr} \\
\C^G \ar[rr]^{\text{Forg}} && \C^H
}
\end{equation}
commutes.
Here $\text{Forg}: \C^G \to \C^H$ is the functor of ``partially forgetting equivariance''.
\end{corollary}
\begin{proof}
By Proposition~\ref{AinC'} there is an \'etale algebra $A$ in $\C_1'$  such
that $\C_2 \cong (\C_1)_A$.
Let $\E =\Rep(G)$ be the maximal Tannakian subcategory of $\C_1'$ and let
$\C=(\C_1)_G$.
Since equivariantization and de-equivariantization
are mutually inverse constructions (see \cite[Theorem 4.4]{DGNO} and
Section~\ref{eq and deeq}) we have
$\C_1\cong \C^G$.

By Remark~\ref{A=FunG/H} there is a subgroup $H \subset G$ such that $A =
\Fun(G/H)$.
Note that a $\Fun(G/H)$-module in $\C_1$ is the same thing as an
$H$-equivariant $\Fun(G)$-module,
which implies  $(\C_1)_A \cong ((\C_1)_G)^H =\C^H$. Furthermore, the
forgetful functor $\C^G \to \C^H$
identifies with the given functor $F: \C_1\to \C_2\cong (\C_1)_A$ since both
of them
correspond to the same \'etale algebra $A=\Fun(G/H)$.
\end{proof}

\begin{definition}
A braided fusion category $\C$ is called {\em almost non-degenerate} if the symmetric category
$\C'$  is either trivial or is equivalent  to the category  of super vector spaces.
\end{definition}

In other words, $\C$ is almost non-degenerate if $\C'$ does not contain any non-trivial Tannakian subcategories.

\begin{corollary}
\label{nondeginj}
Any braided tensor functor $F: \C \to \D$ between braided
fusion categories with $\C$ almost non-degenerate is fully faithful.
\end{corollary}

\begin{remark}  Using \cite[Theorem 2.5]{EO} and \cite[Proposition 2.14]{De}
one can relax the assumptions of Corollary~\ref{nondeginj}
on the category $\D$: it is enough to assume that $\D$ is a abelian rigid braided tensor category
with finite dimensional $\Hom$ spaces and finite lengths of all objects.
\end{remark}

Let $\C$ be a braided fusion category, $A\in \C$ be a connected \'etale algebra
and $F_A: \C \to \C_A$ be the functor \eqref{functor FA} with the central structure $F'_A$
(\ref{tilde functor FA}). The functor
\begin{equation}
\label{functor TA}
T_A: \C_A\boxtimes \C_A^{\rev}\to {}_A \C_A : M\boxtimes N \mapsto  M_+\otimes_AN_-
\end{equation}
has a natural structure of tensor functor.

\begin{corollary}
\label{tilde F}
Assume $\C$ is almost non-degenerate. Then the functor $F'_A$ in \eqref{tilde functor FA} 
is fully faithful and the functor $T_A: \C_A\boxtimes \C_A^{\rev}\to {}_A \C_A$ defined in 
\eqref{functor TA} is surjective.
\end{corollary}
\begin{proof}
The first assertion is Corollary~\ref{nondeginj}.  To prove the second
assertion
observe that that $F'_A$ is dual to $T_A$ (in the sense of \cite[Section
5.7]{ENO1})
with respect to  the module category $\C_A$.  Indeed, an object $M\boxtimes
N$
of $\C_A\boxtimes \C_A^{\rev}$
corresponds to the $\Z(\C_A)$-module endofunctor $M\ot_A - \ot_A N$ of
$\C_A$.
The functor dual to $F'_A$ restricts this endofunctor to the $\C$-module
endofunctor of $\C_A$ by means of $F'_A: \C \to \Z(\C_A)$.  This is
precisely
what $T_A(M\boxtimes N)$ does.
So the result follows from \cite[Proposition 5.3]{ENO1}.
\end{proof}

\subsection{Tensor complements}
Let $\C$ be a non-degenerate braided fusion category, see Definition~\ref{nondegdef}. 
Let $A\in \C$ be a connected \'etale
algebra. Then $A$ can be considered as a connected \'etale algebra in $\C^{\rev}$
and in $\Z(\C)$ via the embedding 
\[
\C^{\rev}= \Vec \boxtimes \C^{\rev}\hookrightarrow \C \boxtimes \C^{\rev}
\cong \Z(\C), 
\]
see \eqref{Gfun}. 

\begin{lemma} \label{ZCA0}
Under the identification $\Z(\C)\simeq \C \boxtimes \C^{\rev}$ we have
\begin{equation*}
\Z(\C)_A=\C \boxtimes \C^{\rev}_A \quad \mbox{and} \quad
\Z(\C)_A^0=\C \boxtimes (\C^{\rev})_A^0.
\end{equation*}
\end{lemma}
\begin{proof} 
The first statement is obvious and the second one is an immediate consequence.
\end{proof}

\begin{corollary} \label{ZCA}
For a non-degenerate $\C$ and a connected \'etale algebra $A\in \C$
there is a braided equivalence $\Z(\C_A)\simeq \C \boxtimes (\C_A^0)^{\rev}$.
In particular the category $\C_A^0$ is non-degenerate.
\end{corollary}

\begin{proof} Combine Theorem \ref{petersch} and Lemma \ref{ZCA0}.
\end{proof}

\begin{remark} \label{fprime}
\begin{enumerate}
\item[(i)] One can show that the embedding functor 
\[
\C =\C \boxtimes \Vec \hookrightarrow \C \boxtimes (\C_A^0)^{\rev}\cong \Z(\C_A)
\]
is naturally isomorphic to the functor $F'_A$ from \eqref{tilde functor FA}, providing
an alternative proof of the injectivity of that functor, as asserted in Corollary \ref{tilde F}.
\item[(ii)] If we assume in addition that $\C$ is modular
and $A$ is as in Example \ref{regularce}(ii) then $\C_A^0$
has a natural spherical structure, see e.g. \cite{KiO}. In this case
Corollary \ref{ZCA} gives an alternative proof of \cite[Theorem 4.5]{KiO}.
\end{enumerate}
\end{remark}

\begin{corollary} For a non-degenerate $\C$ and a connected \'etale algebra $A\in \C$
we have
{\em \begin{equation}
\label{FPdimCA0}
\FPdim(\C_A^0)=\frac{\FPdim(\C)}{\FPdim(A)^2}.
\end{equation}}
\end{corollary}

\begin{proof} This follows immediately from Corollary \ref{ZCA} and equations
 \eqref{dimZ} and \eqref{FPdimCA}.
\end{proof}

\section{Quantum Manin pairs} \label{qmp}

\subsection{Definition of a quantum Manin pair}
\label{defQMp}

We start with the following consequence of Corollary~\ref{tilde F}.

\begin{corollary} \label{equ} Let $\C$ be a non-degenerate braided fusion category and
$A\in \C$ a connected \'etale algebra in $\C$.
Assume that {\em $\FPdim(A)^2=\FPdim(\C)$}.
Then 
\begin{enumerate}
\item[(i)] The functor $F'_A: \C \to \Z(\C_A)$ defined in \eqref{tilde functor FA}
is a braided tensor equivalence
\item[(ii)] The functor $T_A: \C_A\boxtimes \C_A^{\rev}\to {}_A \C_A$ 
defined in \eqref{functor TA} is a tensor equivalence.
\end{enumerate}
\end{corollary}
\begin{proof} By Lemma~\ref{dim CAlemma}
\[
\FPdim(\C_A)=\frac{\FPdim(\C)}{\FPdim(A)}.
\]
Hence,  
\[
\FPdim(\Z(\C_A))=\frac{\FPdim(\C)^2}{\FPdim(A)^2}=\FPdim(\C), 
\]
see \eqref{dimZ}.
Since by Corollary~\ref{tilde F} $F'_A$ is a fully faithful functor between categories
of equal Frobenius-Perron dimension,  it is necessarily  an equivalence by
\cite[Proposition 2.19]{EO}. Hence the dual
functor $T_A$ is also an equivalence.
\end{proof}

\begin{definition}\label{qMp}
 A {\em quantum Manin pair} $(\C,A)$ consists of a non-degenerate braided
fusion category $\C$ and a connected \'etale algebra $A\in \C$ such that $\FPdim(A)^2=\FPdim(\C)$.
\end{definition}

\begin{remark} \label{M2rep0}
Observe that by \eqref{FPdimCA0} the condition $\FPdim(A)^2=\FPdim(\C)$ is equivalent to the
condition $\C_A^0=\Vec$.
\end{remark}

Quantum Manin pairs form a 2-groupoid $\QM$: a 1-morphism between two such pairs $(\C_1,\,A_1)$
and $(\C_2,\,A_2)$ is defined to be a pair $(\Phi,\, \phi)$, where $\Phi: \C_1\simeq \C_2$ is a braided
equivalence and $\phi: \Phi(A_1)\iso A_2$ is an isomorphism of algebras; a 2-morphism between
pairs $(\Phi,\,\phi)$ and $(\Phi',\,\phi')$ is a natural isomorphism of tensor functors
$\mu:\Phi \xrightarrow{\sim} \Phi'$ such that the following diagram commutes:
\begin{equation}
\label{2mor in QP}
\xymatrix{\Phi(A_1)\ar[rr]^{\mu} \ar[rd]_{\phi}&&\Phi'(A_1)\ar[ld]^{\phi'}\\ &A_2&}.
\end{equation}
On the other hand, we have the 2-groupoid $\FC$ of fusion categories: objects are fusion categories,
1-morphisms are tensor equivalences, and 2-morphisms are isomorphisms of tensor
functors. We have a 2-functor $\QM \to \FC$ defined by $(\C,\,A)\mapsto \C_A$.

\begin{proposition}
\label{dva}
This 2-functor $\QM \to \FC$ is a 2-equivalence.
\end{proposition}
\begin{proof} Let $\A\in\FC$. The forgetful functor $F:\Z(\A)\to \A$ has an obvious structure of
central functor. Let $I:\A\to\Z(\A)$ be its right adjoint. By Lemma~\ref{centraletale}, $I(\be)$  is a
connected \'etale algebra. It is known that $\FPdim(I(\be))=\FPdim(\C)$, see e.g. \cite[Lemma 3.41]{EO}.
So \eqref{dimZ} implies that $(\Z(\A),\,I(\be))\in \QM$. Thus we get a 2-functor $\FC\to\QM$.
Using Corollary \ref{equ} and the results from Section~ \ref{eacf} we see that  it is quasi-inverse
to the 2-functor $\QM \to \FC$.
\end{proof}

\begin{remark}
\label{Manin 2}
Proposition \ref{dva} can be viewed as a
categorical analogue of the following reconstruction of the double
of a quasi-Lie bialgebra from a Manin pair (i.e., a pair
consisting of a metric Lie algebra and its Lagrangian subalgebra)
in the theory of quantum groups \cite[Section 2]{Dr}:

Let $\mathfrak{g}$ be a finite-dimensional metric Lie algebra
(i.e., a Lie algebra on which a non-degenerate invariant symmetric
bilinear form is given). Let $\mathfrak{l}$ be a Lagrangian
subalgebra of $\mathfrak{g}$. Then $\mathfrak{l}$ has  a structure
of a quasi-Lie bialgebra and there is an isomorphism  between
$\mathfrak{g}$ and the double $\mathfrak{D}(\mathfrak{l})$ of
$\mathfrak{l}$. The correspondence between Lagrangian subalgebras
of   $\mathfrak{g}$ and  doubles isomorphic to $\mathfrak{g}$ is
bijective, see \cite[Section 2]{Dr} for details.
\end{remark}


\subsection{Lagrangian algebras and module categories}
\label{QM and Morita2}

\begin{definition}
Let $\C$ be a non-degenerate braided fusion category.  A connected \'etale algebra
in $\C$ will be called {\em Lagrangian} if  $\FPdim(A)^2 =\FPdim(\C)$.
\end{definition}

Thus,  $A$ is Lagrangian if and only if $(\C,\, A)$ is a quantum Manin pair.

\begin{remark}
Let $\E \subset \C$  be a Lagrangian subcategory of $\C$, i.e., a Tannakian
subcategory such that $\E' =\E$, see \cite[Definition 4.57]{DGNO}. Then
the regular algebra $A$ of $\E$ is a Lagrangian algebra in $\C$. Indeed, Example \ref{tannaka rep0} 
says that $\C_A^0=\Vec$ and the statement follows from Remark \ref{M2rep0}.
\end{remark}

\begin{proposition}
\label{Lagr<--> module}
Let $\A$ be a fusion category and let $\C=\Z(\A)$. There is a bijection between the sets of Lagrangian
algebras in $\C$ and indecomposable $\A$-module categories.
\end{proposition}
\begin{proof}
By Corollary~\ref{equ}  every Lagrangian algebra $B\in \C$ determines a braided equivalence
$\C \cong \Z(\mathcal{B})$, where $\mathcal{B}:=\C_B$.  Conversely, any braided equivalence
between $\C$ and  $\Z(\mathcal{B})$ determines a surjective central functor $\C \to \mathcal{B}$ and, hence, a connected \'etale algebra $A\in \C$, see Lemma \ref{centraletale}.
Combining Lemma \ref{dim CAlemma} and equation \eqref{dimZ} 
we see that the algebra $A$ is Lagrangian. 
As we observed in Section~\ref{eacf} these two constructions
are inverses of each other.

Thus it suffices to prove that  the set of braided equivalences between $\Z(\A)$ and centers
of fusion categories is in bijection with indecomposable $\A$-module categories. This is done
in \cite[Theorem 3.1]{ENO2} and \cite[Theorem 1.1]{ENO3}. Namely, the bijection is provided
by assigning to an $\A$-module category $\M$  braided equivalence \eqref{Schauenburg}.
\end{proof}

\begin{remark} \label{dualcat}

(i) It follows from the proof that the bijection in Proposition \ref{Lagr<--> module} has the 
following property: for a Lagrangian algebra $B\in \C$ the fusion category $\C_B$ is equivalent
to the dual category $\A_\M^*$ where $\M$ is the module category corresponding to $B$.

(ii) Note that the bijection in Proposition \ref{Lagr<--> module} is given by the so-called full centre 
construction. In particular, $I(\be)$ is the full centre of $\A$ as a module category over itself. In 
the case when $\A$ is modular the statement of the proposition was verified in \cite[Theorem 3.22]{kr}. 
Note also that in this case the bijection can be lifted to an equivalence of groupoids (module 
categories with module equivalences by one side and Lagrangian algebras and isomorphisms by the 
other) \cite{dkr}.  
\end{remark}

\subsection{Lattice of subcategories}

Let $\A$ be a fusion category and let $(\C,\,A)$ be the corresponding Manin pair.
Here $\C =\Z(\A)$ and $A=I(\be)$, where $I: \A \to \Z(\A)$ is the induction functor.

Let $\fL(\A)$ denote the lattice of fusion subcategories of $\A$ and let
$L(A)$ denote the lattice of  \'etale subalgebras of $A$.

\begin{theorem}
\label{Lattice iso}
There is a canonical anti-isomorphism of lattices $\fL(\A)\simeq L(A)$.
If $B\subset A$ is the subalgebra corresponding to the subcategory $\B \subset \A$
under this anti-isomorphism, then $\FPdim(B)\FPdim(\B) =\FPdim(\A)$.
\end{theorem}
\begin{proof}
We will construct mutually inverse order-reversing bijections $\alpha: \fL(\A)\to L(A)$
and $\beta: L(A)\to \fL(\A)$.

Let $\B \subset \A$ be a fusion subcategory. 
Define the {\em relative center}  $\Z_\B(\A)$ to be the tensor category
whose objects are
pairs $(X, \gamma_X)$, where $X$ is an object of $\A$ and
$\gamma_X : V\ot X
\simeq X \ot V$, $V\in\B$
is a natural family of isomorphisms, satisfying the same
compatibility condition as in the definition of $\Z(\A)$.
The forgetful functor $\Z(\A)\to \A$ has a  factorization 
$\Z(\A)\xrightarrow{F_\B}\Z_\B(\A)\xrightarrow{\tilde F_\B} \A$ where $F_\B$ and $\tilde F_\B$ are
the obvious forgetful functors. Let $I_\B$ and $\tilde I_\B$ be the right adjoint functors of $F_\B$ and
$\tilde F_\B$. The embedding $\be \subset \tilde I_\B(\be)$ corresponding to the identity map
under the isomorphism $\Hom(\be ,\tilde I_\B(\be))=\Hom(\tilde F_\B(\be),\be)$ induces an
embedding of algebras $I_\B(\be)\subset I_\B\circ \tilde I_\B(\be)=I(\be)=A$. The algebra 
$I_\B(\be)$ is separable (and hence \'etale), see Remark \ref{brn}. We define 
\begin{equation}
\label{FB}
\alpha(\B)=I_\B(\be)\subset A.
\end{equation}

An inclusion of subcategories $\B_1\subset \B_2 \subset \A$ induces a
factorization $\Z(\A)\xrightarrow{F_{\B_2}}\Z_{\B_2}(\A) \to \Z_{\B_1}(\A)$ of the functor $F_{\B_1}$.
This, in turn, yields an inclusion of subalgebras $I_{\B_2}(\be) \subset I_{\B_1}(\be)\subset A$. Thus
the map $\alpha$ is order-reversing.

The functor $F_\B$ is surjective by \cite[Section 3.6]{DGNO}. 
Hence we have
\[
\FPdim(\alpha(\B)) = \frac{\FPdim(\A)}{\FPdim(\B)}
\]
by (the proof of) \cite[Corollary 8.11]{ENO1}.

In the opposite direction, given an \'etale subalgebra $B \subset A$
we have a tensor functor $?\ot_BA: \C_B \to \C_A$ inducing $A$-modules
from $B$-modules.  Let $\beta(B)$ be the full image in $\C_A=\A$
of the subcategory $\C_B^0 \subset \C_B$ under this functor. Observe that $A$ considered 
as a $B-$module is dyslectic. It follows that the objects of $\beta(B)$ are precisely $A-$modules 
which are dyslectic as $B-$modules. This implies that the map $\beta$ is order-reversing. 
Observe that the right adjoint functor of $?\ot_BA$ is isomorphic to the forgetful functor
$\C_A\to \C_B$ and sends the unit object of $\C_A$ to  $A\in \C_B^0\subset \C_B$. Using again
the proof of \cite[Corollary 8.11]{ENO1} we see that
\begin{equation}
\label{FeqP}
\FPdim(\beta(B))= \frac{\FPdim(A)}{\FPdim(B)}.
\end{equation}

By construction, $\C_{\alpha(\B)} =\Z_\B(\A)$. We claim that the subcategory 
$\C_{\alpha(\B)}^0\subset \C_{\alpha(\B)}$ identifies with $\Z(\B)\subset \Z_\B(\A)$.
Indeed, by Corollary \ref{ZCA} the category $(\C_{\alpha(\B)}^0)^\rev$ identifies with the centralizer
of $\C$ in $\Z(\C_{\alpha(\B)})$. On the other hand it is explained in \cite[Section 3.6]{DGNO} that
$\Z_\B(\A)=(\A \bt \B^\op)^*_\A$ (see Section \ref{dualmcat} for the notations), 
so equation \eqref{Schauenburg} implies
$\Z(\Z_\B(\A))=\Z(\A)\bt \Z(\B)^\rev$. The central functor 
$$\Z(\A)=\Z(\A)\bt \be \subset \Z(\A)\bt \Z(\B)^\rev =\Z(\Z_\B(\A))\to \Z_\B(\A)$$
identifies with $F_\B$ with obvious central structure. Hence
the subcategories $\Z(\A)=\Z(\A)\bt \be \subset \Z(\Z_\B(\A))$
and $\C \subset \Z(\C_{\alpha(\B)})$ coincide and so do their centralizers in $\Z(\Z_\B(\A))$
and their images in $\Z_\B(\A)= \C_{\alpha(\B)}$. Our claim follows.

The induction functor
\begin{equation}
\label{AB induction}
\C_{\alpha(\B)}\to \C_A = \A
\end{equation}
identifies with the forgetful functor $\Z_\B(\A) \to \A$ and so maps
surjectively $\Z(\B)=\C_{\alpha(\B)}^0$ to $\B$. Thus, $\beta(\alpha(\B))=\B$.

Conversely, we claim that there is an equivalence $\Z_{\beta(B)}(\A) \xrightarrow{\sim}\C_B$
such that the forgetful functor $F_{\beta(B)}: \Z(\A) \to \Z_{\beta(B)}(\A)$
identifies with the free module functor $\C \to \C_B$.  This immediately
implies that $\alpha(\beta(B))=B$.  To prove this claim, note
that the braiding of $\C$ allows to equip any $A$-module
induced from $\C_B$ with a morphism permuting it with the objects of $\beta(B)\subset \C_A$
(notice that for $M\in \C_B$ and $N\in \beta(B)$ we have $(M\ot_BA)\ot_AN=M\ot_BN$ and
$N\ot_A(M\ot_BA)=N\ot_BM$).
This gives rise to a tensor functor
\begin{equation}
\label{CBZBA}
F''_A: \C_B \to \Z_{\beta(B)}(\C_A), \quad M \mapsto  M\ot_BA.
\end{equation}
Recall the equivalence $F'_A$ from Corollary \ref{equ} (i). 
It follows from the above definition that the diagram 
\begin{equation}
\xymatrix{
\C \ar[rrr]^{F'_A} \ar[d]_{?\ot B}  &&& \Z(\C_A) \ar[d]^{F_{\beta(B)}} \\
\C_B \ar[rrr]^{F''_A} &&&  \Z_{\beta(B)}(\C_A)
}
\end{equation}
commutes. In particular, the induction functor \eqref{CBZBA} is surjective.
Using \cite[equation (56)]{DGNO} and equation \eqref{FeqP} we get
\[
\FPdim( \Z_{\beta(B)}(\A)) = \FPdim(\beta(B))\FPdim(\A)= \frac{\FPdim(A)^2}{\FPdim(B)} =\FPdim(\C_B).
\]
Thus functor \eqref{CBZBA} is an equivalence by \cite[Proposition 2.20]{EO}. This completes our proof.


\end{proof}

%
%
%

\begin{example}
Let us illustrate Theorem~\ref{Lattice iso}. Let $G$ be a finite group.
\begin{enumerate}
\item[(i)] Let $\A =\Rep(G)$ be the fusion category of  representations of $G$.
Its fusion subcategories are of the form $\Rep(G/N)$ where $N$ ranges over the set
of all normal subgroups of $G$. The \'etale algebra in $\Z(\Rep(G))$ corresponding
to  the subcategory $\Rep(G/N)$ is the group algebra $\kk N$. As an object of $\Z(\Rep(G))$
it  has the following description.
It is a $G$-graded algebra with  non-zero graded components labelled by elements of $N$,
the $G$-action on $\kk N$ is the conjugation action (see \cite{Da1}, where \'etale algebras in $\Z(\Rep(G))$ were classified).

\item[(ii)]  Let $\A =\Vec_G^\omega$ be the fusion category of $G$-graded
vector spaces with the associativity constraint twisted by a $3$-cocycle
$\omega\in Z^3(G,\, \kk^\times)$.  Fusion subcategories of $\A$ correspond to
subgroups $H \subset G$. A typical such subcategory  is $\Vec_H^{\omega|_H}$.
The corresponding \'etale algebra in $\Z(\Vec_G^\omega)$ is
the algebra of  $\kk$-valued  functions on $G$ invariant under translations
by elements of $H$.
\end{enumerate}
\end{example}


\begin{remark} Let $\C$ be a non-degenerate braided fusion category and let $A\in \C$ be
a connected \'etale algebra. Recall that $\Z(\C_A)\simeq \C \bt (\C^\rev)_A^0$ (see Corollary \ref{ZCA})
and the functor $\C =\C \bt \be \subset \Z(\C_A)\to \C_A$ is isomorphic to the free module
functor $F_A$, see Remark \ref{fprime} (i).  It follows that $A=A\bt \be \in \Z(\C_A)$ is a subalgebra of the
Lagrangian algebra $I(\be)$. It is easy to see that the corresponding subcategory of $\C_A$ is
precisely $\C_A^0$. Thus Theorem \ref{Lattice iso} implies the following statement: the lattice of 
subalgebras of $A$ is anti-isomorphic to the lattice of subcategories of $\C_A$ containing
$\C_A^0$. Notice that Theorem \ref{Lattice iso} is a special case of this statement, 
see Remark \ref{M2rep0}.
\end{remark} 

\subsection{Quantum Manin triples}

Recall that a {\em Manin triple} consists of a metric Lie algebra $\mathfrak{g}$ along
with Lagrangian Lie subalgebras $\mathfrak{g}_+,\, \mathfrak{g}_-$ such
that $\mathfrak{g} = \mathfrak{g}_+\oplus \mathfrak{g}_-$ as a vector space.
It was shown by Drinfeld in \cite[Section 2]{Dr} that Manin triples are in bijection
with pairs of dual Lie bialgebras (cf.\ Remark~\ref{Manin 2}).

Below we extend this result to the ``quantum'' setting.

\begin{definition}
\label{quantum Manin 3}
A {\em quantum Manin triple} $(\C, A,\, B)$ consists of a non-degenerate braided fusion category
$\C$ along with connected \'etale algebras $A,\,B$ in $\C$ such that both $(\C,\, A)$ and $(\C,\,B)$ are
quantum Manin pairs
and  the category of $(A,B)$-bimodules in $\C$  is equivalent to $\Vec$.
\end{definition}

\begin{example}
\label{Hopf --> M3}
Let $H$ be a semisimple Hopf algebra and let $\Rep(H)$ denote the category
of finite-dimensional representations of $H$. Let $\C:= \Z(\Rep(H))$. It is well known
that $\C$ is equivalent, as a braided fusion category,  to $\Rep(D(H))$ where
$D(H)$ is the Drinfeld double of $H$. There is a canonical Hopf algebra
isomorphism $D(H) \cong D((H^*)^{op})$, where $H^*$ denotes the dual Hopf algebra
and $op$ stands for the opposite multiplication. We thus
have two central functors, to wit the forgetful functors,
\begin{equation*}
 \C \to \Rep(H)\quad \mbox{and} \quad \C \to \Rep((H^*)^{op}).
\end{equation*}
Let $A$ and $B$ denote the \'etale algebras in $\C$ corresponding to these functors
constructed as in Section~\ref{catCA1}.

We claim that $(\C,\, A,\, B)$ is a quantum Manin triple.  The only thing that needs to be checked
is that the category of $(A,B)$-bimodules in $\C$ is trivial. Note that $A= (H^*)^{op}$
and $B = H$ as $D(H)$-module algebras (i.e., algebras in $\C =\Rep(D(H))$).
The category of $(H^*)^{op} \ot H$-bimodules in $\Rep(D(H))$ is nothing but
the category of $D(H)$-Hopf modules  which is equivalent to $\Vec$
by the Fundamental Theorem of Hopf modules (see \cite{Mo}  for the definition  of a Hopf
module and the Fundamental Theorem).
\end{example}

We explain now that any quantum Manin triple arises from the construction in Example \ref{Hopf --> M3}.
Let $(\C,\, A,\, B)$ be a quantum Manin triple. Then $\Vec$ identified with $(A,B)$-bimodules has a structure of a $\C_A$-module category via $\ot_A$. Equivalently, $\C_A$ has a fiber functor, 
i.e., a tensor functor to $\Vec$, see \cite[Proposition 4.1]{O}. 
Thus $\C_A \cong \Rep(H)$ for a semisimple Hopf algebra $H$, see \cite{Ul}.
The dual category $(\C_A)^*_\Vec$ is equivalent to $\C_B$ (see Remark \ref{dualcat} (i))
and so $\C_B \cong \Rep((H^*)^{op})$, see \cite[Theorem 4.2]{O}.

Quantum Manin triples form a $2$-groupoid $\mathcal{G}_1$:   a 1-morphism between triples  $(\C_1,\,A_1,\, B_1)$
and $(\C_2,\,A_2\,B_2)$ is defined to be a triple $(\Phi,\, \phi,\,\psi)$, where $\Phi: \C_1\simeq \C_2$ is a braided
equivalence and $\phi: \Phi(A_1)\iso A_2,\, \psi: \Phi(B_1)\iso B_2$ are isomorphisms
of algebras; a 2-morphism between triples $(\Phi,\, \phi,\,\psi)$ and  $(\Phi',\, \phi',\,\psi')$
is a natural isomorphism of tensor functors
$\mu:\Phi \simeq \Phi'$ such that $\phi =\phi'\, \mu_{A_1}$ and $\psi =\psi'\, \mu_{B_1}$
(cf.\ diagram~\eqref{2mor in QP}).

Let $\mathcal{G}_2$ denote the $2$-groupoid whose objects are pairs $(\A,\, F)$ where $\A$
is a fusion category and $F: \A \to \Vec$ is a fiber functor; $1$-morphisms
between $(\A,\, F)$ and $(\A',\, F')$ are pairs $(\iota,\, \nu)$ where $\iota: \A \iso \A'$
is a tensor equivalence and $\nu~:~F \iso F'\iota$ is an isomorphism of tensor functors;
$2$-morphisms between  $(\iota_1,\, \nu_1)$  and $(\iota_2,\, \nu_2)$ are natural
isomorphisms  of tensor functors $m: \iota_1\iso \iota_2$ such that  $\nu_2= (F'm)\circ \nu_1$.

As we explained above a quantum Manin triple $(\C,\, A,\, B)$ gives rise to a fusion category
$\C_A$ equipped with a fiber functor $F: \C_A\to \Vec$. This construction can be upgraded to a 
2-functor $\mathcal{G}_1\to \mathcal{G}_2$. Similarly, the construction from Example \ref{Hopf --> M3}
can be upgraded to a 2-functor $\mathcal{G}_2\to \mathcal{G}_1$ (we recall that by \cite{Ul} a pair
$(\A,F)\in \mathcal{G}_2$ is isomorphic to the pair $(\Rep(H), F_H)$ where $H$ is a semisimple
Hopf algebra and $F_H: \Rep(H)\to \Vec$ is the forgetful functor).

\begin{proposition}
\label{tri}
The 2-functors above are mutually inverse 2-equivalences
between $\mathcal{G}_1$ and $\mathcal{G}_2$.
\end{proposition}

The proof  of Proposition~\ref{tri} is similar to that
of Proposition~\ref{dva} and amounts to showing that the above constructions are inverses of each other.
In fact,  $2$-groupoids   $\mathcal{G}_1$ and $\mathcal{G}_2$
are also equivalent to the third $2$-groupoid
$\mathcal{G}_3$ which is defined in linear algebra terms:  objects of $\mathcal{G}_3$ are
semisimple Hopf algebras, 1-morphisms are twisted isomorphisms of Hopf algebras (defined
in \cite{Da}), and $2$-morphisms are gauge equivalences of twists.  Details of these
equivalences will be given elsewhere.

Finally, we give an easy criterion which allows to recognize a quantum Manin triple.
Let $R_\C \in K(\C)\ot_\mathbb{Z} \BR$ denote the regular object of $\C$, see Section~\ref{fuscat}.

\begin{proposition}
\label{4 conditions on triple}
Let $\C$ be a non-degenerate braided fusion category and let $(\C,\, A)$,
$(\C,\,B)$ be quantum Manin pairs. The following conditions are equivalent:
\begin{enumerate}
\item[(i)] $(\C,\,A,\,B)$ is a quantum Manin triple;
\item[(ii)]  $[A\ot B]=R_\C$;
\item[(iii)]  $\dim_\kk \Hom_\C(\be,\, A\ot B)=1$;
\item[(iv)]  $\dim_\kk \Hom_\C(A,\,B)=1$.
\end{enumerate}
\end{proposition}
\begin{proof}
Let us prove implication (i) $\Rightarrow$ (ii). Thus the category of $(A,B)-$bimodules
has a unique up to isomorphism simple object $M$. For any $X\in \C$, the object $A\ot X\ot B$ has
an obvious structure of $(A,B)-$bimodule. Hence $[A\ot X\ot B]=r_X[M]$ for some positive integer
$r_X$. Consequently 
\[
[A\ot X\ot B]=\frac{r_X}{r_\be}[A\ot B].
\]
Computing the Frobenius-Perron dimension of both sides, 
we get $[A\ot X\ot B]=\FPdim(X)[A\ot B]$. Since the category $\C$ is braided we have
\[
[X][A\ot B]=[A\ot X\ot B]=\FPdim(X)[A\ot B].
\]
Since
$\FPdim(A)=\FPdim(B)=\sqrt{\FPdim(\C)}$, we have $\FPdim(A\ot B)=\FPdim(\C)$. Hence
$[A\ot B]=R_\C$, see Section \ref{fuscat}.

The implication (ii) $\Rightarrow$ (iii) is immediate and the equivalence (iii) $\Leftrightarrow$ (iv)
follows from Remark \ref{etsd} since $\Hom_\C(A,\,B)=\Hom_\C(\be,\,^*A\ot B)\simeq \Hom_\C(\be,\, A\ot B)$.

Let us prove implication (iii) $\Rightarrow$ (i). By Corollary \ref{equ} (i), the central functor
$F_B: \C \to \C_B$ is isomorphic to the forgetful functor $\Z(\C_B)\to \C_B$ (for a suitable choice of
braided equivalence $\C \simeq \Z(\C_B)$). Consider the category $\Rep_{\C_B}(A)$ (see
Section~ \ref{schauenburg}). Notice that by Remark \ref{left=right}(ii), this category coincides with the
category of $(A,B)-$bimodules in $\C$. Thus, we need to prove that $\Rep_{\C_B}(A)\simeq \Vec$.
Recall from Section~ \ref{schauenburg} that the category $\Rep_{\C_B}(A)$ has a structure of
multi-fusion category. On the other hand the unit object $A\ot B$ of this category
is irreducible since $\Hom_{A-B}(A\ot B,\,A\ot B)=\Hom_\C(\be,\, A\ot B)$. Thus, the multi-fusion
category $\Rep_{\C_B}(A)$ is in fact a fusion category. By Theorem \ref{petersch} and
Remark \ref{M2rep0} we have
$\Z(\Rep_{\C_B}(A))=\C_A^0=\Vec$. Thus \eqref{dimZ} implies that $\FPdim(\Rep_{\C_B}(A))=1$,
whence $\Rep_{\C_B}(A)=\Vec$.
\end{proof}




\section{Definition and properties of the Witt group}
\subsection{Definition of the Witt group}
\begin{definition} 
\label{Witt eq}
Non-degenerate braided fusion categories $\C_1$ and $\C_2$
are {\em Witt equivalent} if there exists a braided equivalence
$\C_1\boxtimes \Z(\A_1)\simeq \C_2\boxtimes \Z(\A_2)$, where $\A_1$, $\A_2$ are fusion categories.
\end{definition}

\begin{remark} \label{multiW}
The equivalence relation in Definition \ref{Witt eq} will not change if we allow $\A_1$ and $\A_2$ 
to be non-zero multi-fusion categories. Indeed, assume that $\C_1\boxtimes \Z(\A_1)\simeq \C_2\boxtimes \Z(\A_2)$ where $\A_1$ and $\A_2$ are multi-fusion categories. We can assume
that $\A_1$ and $\A_2$ are indecomposable in the sense of \cite[Section 2.4]{ENO1} (replace 
$\A_1$ and $\A_2$ by suitable summands otherwise). It follows from 
\cite[Lemma 3.24, Corollary 3.35]{EO} that for an indecomposable multi-fusion category
$\A$ there exists a fusion category $\A'$ and a braided equivalence $\Z(\A)\simeq \Z(\A')$.
Our statement follows.
\end{remark} 

It is easy to see that  Witt equivalence is indeed an equivalence relation.
For example the transitivity holds since the conditions 
$\C_1\boxtimes \Z(\A_1)\simeq \C_2\boxtimes \Z(\A_2)$
and $\C_2\boxtimes \Z(\A_2')\simeq \C_3\boxtimes \Z(\A_3)$ imply
$$\C_1\boxtimes \Z(\A_1 \bt \A_2')\simeq \C_2\boxtimes \Z(\A_2\bt \A_2')\simeq 
\C_2\boxtimes \Z(\A_2'\bt \A_2)\simeq \C_3\bt \Z(\A_3 \bt \A_2).$$
We will denote the Witt equivalence class containing a category $\C$ by $[\C]$.
The set of Witt equivalence classes of non-degenerate braided fusion categories
will be denoted $\W$. Clearly $\W$ is a commutative monoid with respect to the
operation $\boxtimes$. The unit of this monoid is $[\Vec]$.

\begin{lemma} \label{Wisgroup}
The monoid $\W$ is a group.
\end{lemma}
\begin{proof} For a non-degenerate braided fusion category $\C$ we have
$\Z(\C)\simeq \C \boxtimes \C^{\rev}$, see Section~ \ref{dcenter}. Thus $[\C]^{-1}=[\C^{\rev}]$.
\end{proof}

\begin{proposition}
\label{rep0} 
Let $A\in \C$ be an \'etale connected algebra. Then $[\C_A^0] = [\C]$ in $\W$.
\end{proposition}
\begin{proof}
This is immediate from Definition~\ref{Witt eq}, Lemma \ref{Wisgroup} and  Corollary \ref{ZCA} .
\end{proof}


%

\begin{definition} 
The abelian group $\W$ defined above is called {\em the Witt group of non-degenerate
braided fusion categories.}
\end{definition}

\begin{remark}\label{galois}
It is apparent from the definition that the group $\W$ depends on the base field $\kk$ and should be
denoted $\W(\kk)$. However it is known that any fusion category (or braided fusion category) is defined
over the field of algebraic numbers $\bar \BQ$, see \cite[Section~ 2.6]{ENO1}. Thus an embedding
$\bar \BQ \subset \kk$ induces an isomorphism $\W(\bar \BQ)\simeq \W(\kk)$. In this sense we can
talk about {\it the} Witt group of non-degenerate braided fusion categories (without mentioning 
the field $\kk$). Of course this implies that the group $\W$ carries a natural action of the absolute
Galois group $Gal(\bar \BQ/\BQ)$ and should be considered together with this action.
\end{remark}

\begin{remark} 
\label{Wcount}
It follows from \cite[Theorems 2.28, 2.31, and Remark 2.33]{ENO1} that there
are countably many non-equivalent braided fusion categories. In particular,
the group $\W$ is at most countable. We will see later that $\W$ is infinite.
\end{remark}

\begin{proposition} 
\label{whenCisinVec}
Let $\C$ be a non-degenerate braided fusion category.
Then $\C \in [\Vec]$ if and only if there exist a fusion
category $\A$ and a braided equivalence $\C \cong \Z(\A)$.
\end{proposition}
\begin{proof} 
By definition, $\C \in [\Vec]$ if and only if $\C \boxtimes \Z(\B_1)\simeq \Z(\B_2)$
with fusion categories $\B_1$ and $\B_2$. By Proposition \ref{dva} there exists 
a connected \'etale algebra $A\in \Z(\B_1)$
such that $(\Z(\B_1),A)$ is a quantum Manin pair, see Definition~\ref{qMp}. 
By abuse of notation we will denote by $A\in \Z(\B_2)$ the image of $\be \bt A$
under the equivalence $\C \boxtimes \Z(\B_1)\simeq \Z(\B_2)$.
Consider the 
multi-fusion category
$\A =\Rep_{\B_2}(A)$, see Section~ \ref{schauenburg}. By Theorem~\ref{petersch} we have $\Z(\A)\cong
\Z(\B_2)^0_A$. On the other hand we have an obvious injective braided tensor functor 
\begin{equation}
\label{CtoZB2}
\C \to \Z(\B_2)^0_A : X\mapsto (X\boxtimes \be)\ot A.
\end{equation}
We have
\[
\FPdim(\C)=\frac{\FPdim(\Z(\B_2))}{\FPdim(\Z(\B_1))}=
\frac{\FPdim(\Z(\B_2))}{\FPdim(A)^2}=\FPdim(\Z(\B_2)_A^0),
\]
i.e.,  \eqref{CtoZB2} is a fully faithful tensor functor between
fusion categories of equal Frobenius-Perron dimension. Therefore,
it is an equivalence by \cite[Proposition 2.19]{EO}. The Proposition follows,
see Remarks \ref{indmult} and \ref{multiW}.
\end{proof}

\begin{corollary}\label{easywitt}
 We have $[\C]=[\D]$ if and only if there exists a fusion category $\A$
and a braided equivalence $\C \boxtimes \D^{\rev}\simeq \Z(\A)$.
\end{corollary}

\subsection{Completely anisotropic categories}

\begin{definition} 
We say that a non-degenerate braided fusion category is {\em
completely anisotropic} if the only connected \'etale algebra $A\in \C$ is $A=\be$.
\end{definition}

\begin{remark}
A completely anisotropic non-degenerate braided fusion category has no
Tannakian subcategories other than $\Vec$, i.e., it is {\em anisotropic}
in the sense of \cite[Definition 5.16]{DGNO}.
\end{remark}

\begin{lemma} \label{caninj}
 Let $\C$ be a completely anisotropic category, $\A$ be a fusion category,
and let $F: \C \to \A$ be a central functor. Then $F$ is fully faithful.
\end{lemma}

\begin{proof} Let $I: \A \to \C$ be the right adjoint of $F$. Since $\C$ is completely
anisotropic, Lemma \ref{centraletale} implies  that $I(\be)=\be$. Thus
\begin{eqnarray*}
\Hom_\C(X,Y) &\cong& \Hom_\C(X\ot {}^*Y,\be) \cong \Hom_\C(X\ot {}^*Y,I(\be))\\
&\cong& \Hom_\A(F(X\ot {}^*Y),\be)  \cong \Hom_\A(F(X)\ot {}^*F(Y),\be) \\
&\cong& \Hom_\A(F(X),F(Y)).
\end{eqnarray*}
The result follows.
\end{proof}

We will say that a connected \'etale algebra $A$ in a braided fusion category $\C$ is {\em maximal}
if it is not a proper subalgebra of another such algebra. For any $\C$ there exists at least one 
maximal connected \'etale algebra since by \eqref{FPdimCA} the Frobenius-Perron dimensions of 
connected \'etale algebras are bounded by $\FPdim(\C)$.

\begin{theorem}
\label{unica}
Each Witt equivalence class in $\W$ contains a completely anisotropic category that is unique up to 
braided equivalence.
\end{theorem}
\begin{proof} 
Let $\C$ be a non-degenerate braided fusion category.
Let $A\in \C$ be a maximal connected \'etale algebra 
By Proposition \ref{overetale} any connected \'etale algebra in $\C_A^0$ can be considered 
as a connected \'etale algebra in $\C$, so maximality
of $A$ is equivalent to $\C_A^0$ being completely anisotropic. Thus, Proposition~\ref{rep0}
implies that any Witt equivalence class contains a completely anisotropic category.

Now let $\C$ and $\D$ be two completely anisotropic categories such that $[\C]=[\D]$.
By Corollary \ref{easywitt} there exists a fusion category $\A$ and a braided equivalence
$\C \boxtimes \D^{\rev}\simeq \Z(\A)$. In particular we have central functors $\C \to \A$
and $\D^{\rev}\to \A$. By Lemma \ref{caninj} these functors are fully faithful. Hence
$\FPdim(\C)\le \FPdim(\A)$ and $\FPdim(\D)\le \FPdim(\A)$. Combining this with \eqref{dimZ}
we see that $\FPdim(\C)=\FPdim(\D)=\FPdim(\A)$ and the functor $\C \to \A$ (and $\D^{\rev}\to \A$)
is an equivalence. 
In particular $\A$ acquires a structure (in fact, two structures) of non-degenerate braided fusion category.
Let $\C'$ be the centralizer of $\C$ in $\C \bt \D^{\rev}\simeq \Z(\A)\simeq \Z(\C)$. Then on one hand 
$\C'=\D^{\rev}$ and on the other hand $\C'=\C^{\rev}$, see Section~\ref{dcenter}. The result follows.
\end{proof}

\begin{corollary} Let $A$ and $B$ be two maximal connected \'etale algebras in a non-degenerate
braided fusion category $\C$. Then there exists a braided equivalence $\C_A^0\simeq \C_B^0$.
In particular {\em $\FPdim(A)=\FPdim(B)$}.
\end{corollary}

\begin{proof} The first statement is immediate from Theorem~\ref{unica}. The second one
follows from \eqref{FPdimCA0}.
\end{proof}

The following result shows that Witt equivalence can also be understood without reference to the Drinfeld center:

\begin{proposition}
Let $\C_1, \C_2$ be non-degenerate braided fusion categories. Then the following are equivalent:
\begin{itemize}
\item[(i)] $[\C_1]=[\C_2]$, i.e.\ $\C_1$ and $\C_2$ are Witt equivalent.
\item[(ii)] There exist a braided fusion category $\C$, connected \'etale algebras $A_1,A_2\in\C$ and braided 
equivalences $\C_1\stackrel{\simeq}{\rightarrow} \C_{A_1}^0,\ \C_2\stackrel{\simeq}{\rightarrow} \C_{A_2}^0$. 
\item[(iii)] There exist connected \'etale algebras $A_1\in\C_1,A_2\in\C_2$ and a braided equivalence
$(\C_1)_{A_1}^0\stackrel{\simeq}{\rightarrow}(\C_2)_{A_2}^0$.
\end{itemize}
\end{proposition}
\begin{proof}
The implications (ii)$\Rightarrow$(i) and (iii)$\Rightarrow$(i) are immediate by Proposition \ref{rep0}.

(i)$\Rightarrow$(ii): By Definition~\ref{Witt eq}, we have a braided equivalence
$$F:\C_1\boxtimes \Z(\A_1)\simeq \C_2\boxtimes \Z(\A_2).$$ Thus we can define $\C$ to be $\C_2\boxtimes \Z(\A_2)$, 
the algebra $A_1$ to be $F(\be \boxtimes I_1(\be))$ and the algebra $A_2$ to be 
$\be \boxtimes I_2(\be)$. 
Here $I_i:\A_i\to \Z(\A_i)$ are right adjoints to the forgetful functors $\Z(\A_i)\to\A_i$. 
Finally we define the braided equivalence $\C_1\to \C_{A_1}^0$ as
$$\C_1\to \C_1\boxtimes \Z(\A_1)_{I_1(\be)}^0\stackrel{F}{\longrightarrow} (\C_2\boxtimes \Z(\A_2))_{A_1}^0 = \C_{A_1}^0$$ 
and the braided equivalence $\C_2\to \C_{A_2}^0$ as
$$\C_2\to \C_2\boxtimes \Z(\A_2)_{I_2(\be)}^0 = \C_{A_2}^0.$$

(i)$\Rightarrow$(iii) Choose \'etale algebras $A_i\in\C_i$ such that the categories $(\C_i)_{A_i}^0$ are completely
anisotropic. Now $[(\C_1)_{A_1}^0]=[\C_1]=[\C_2]=[(\C_2)_{A_2}^0]$ together with Theorem \ref{unica} implies the 
existence of a braided equivalence $(\C_1)_{A_1}^0\stackrel{\simeq}{\rightarrow}(\C_2)_{A_2}^0$.
\end{proof}

\begin{remark}
1. The proposition implies that Witt equivalence is the equivalence relation $\sim$ on non-degenerate braided fusion 
categories generated by ordinary braided equivalence $\simeq$ and the relations $\C\sim\C_A^0$, where  $A\in\C$
is an \'etale algebra. But the proposition is more precise in that it says that any two Witt equivalent categories
can be joined by just two invocations of $\C\sim\C_A^0$ and either one (part (iii)) or two (part (ii))
braided equivalences.

2. The proposition has applications to conformal field theory, cf.\ \cite{mu6}.
\end{remark}

\subsection{The Witt group of metric groups and pointed categories} 
Recall that a {\em quadratic form} with values in $\kk^\times$ on a finite abelian group $A$ is a function
$q: A\to \kk^\times$ such that  $q(-x)=q(x)$ and $b(x,y)=\frac{q(x+y)}{q(x)q(y)}$ is bilinear, see
e.g. \cite[Section~ 2.11.1]{DGNO}. The pair $(A,\,q)$ consisting of finite abelian group and quadratic form
$q: A\to \kk^\times$ is called a {\em pre-metric group}, see \cite[Section~ 2.11.2]{DGNO}. A pre-metric group
$(A,\,q)$ is called {\em metric group} if the form $q$ is non-degenerate (i.e., the associated 
bilmultiplicative form $b(x,y)$ is non-degenerate). 

To a pre-metric group $(A,\,q)$ one assigns a unique up to a braided equivalence 
pointed braided fusion category $\C(A,\,q)$, where  $q(a)\in \kk^\times$ equals 
the braiding on the simple object $X_a\ot X_a$ where $X_a$ 
is a representative of an isomorphism class $a\in A$
(see e.g., \cite[Section~ 2.11.5]{DGNO}). It was shown in \cite{JS2} that this assignment is
an equivalence between the 1-categorical truncation of the 2-category 
of pre-metric groups and that of the 2-category of pointed braided fusion categories.

The category $\C(A,\,q)$ is non-degenerate if and only if $(A,\,q)$ is a metric group, 
see \cite[Sections 2.11.5 and 2.8.2]{DGNO}.

Let $(A,\,q)$ be a metric group and let $H\subset A$ be an {\em isotropic} subgroup (that is, $q|_H=1$). Then $H\subset H^\perp$
where $H^\perp$ is the orthogonal complement of $H$ in $A$ with respect to the bilinear form $b(x,y)$. Moreover, the restriction
of $q$ to $H^\perp$ is the pull-back of a non-degenerate quadratic form $\tilde q: H^\perp /H\to \kk^\times$. We say that
$(H^\perp /H,\tilde q)$ is an {\em m-subquotient} of $(A,\,q)$. Two metric groups are {\em Witt equivalent} if they have
isomorphic m-subquotients (for some choice of isotropic subgroups in each of them), cf.\ \cite[Appendix A.7.1]{DGNO}. The set of equivalence classes has a natural structure of abelian group
(with addition induced by the orthogonal direct sum) and is called the {\em Witt group of metric groups}, see {\em loc.\ cit.} 
We will denote this group $\W_{pt}$. 

\begin{proposition} 
The assignment 
\begin{equation}
\label{Wpt}
\W_{pt}\to \W : (A,\,q)\mapsto [\C(A,\,q)]
\end{equation}
induces a well defined injective homomorphism $\W_{pt}\to \W$.
\end{proposition} 
\begin{proof} 
Let $H\subset A$ be an isotropic subgroup. Then the corresponding subcategory 
$\C(H,1)\subset \C(A,\,q)$ is Tannakian, see e.g. \cite[Example 2.48]{DGNO}. Let $B\in \C(H,1)$ be
the corresponding regular algebra, see \ref{regular}. Then the category $\C(A,\,q)_B^0$ identifies
with $\C(H^\perp/H,\tilde q)$. In particular, $[\C(A,\,q)]=[\C(H^\perp/H,\tilde q)]$.
This implies that \eqref{Wpt} is well defined. 

It is known (see \cite[Section A.7.1]{DGNO}) that each class in $\W_{pt}$ has 
a representative $(A,\,q)$ which is anisotropic, that is
$q(x)\ne 1$ for $A\ni x\ne 1$. It is clear that the corresponding category $\C(A,\,q)$ is completely
anisotropic. Thus, \eqref{Wpt} is injective by Theorem~\ref{unica}.
\end{proof}

In what follows we will identify the group $\W_{pt}$ with its image in $\W$.
The group $\W_{pt}$ is explicitly known, see e.g., \cite[Appendix A.7]{DGNO}. Namely, 
\[
\W_{pt}=\bigoplus_{p\text{ is prime}}\, \W_{pt}(p),
\]
where $\W_{pt}(p)\subset \W_{pt}$
consists of the classes of metric $p-$groups. 

The group $\W_{pt}(2)$ is isomorphic to $\BZ/8\BZ \oplus
\BZ/2\BZ$; it is generated by two classes $[\C(\BZ/2\BZ,\,q_1)]$ and $[\C(\BZ/4\BZ,\,q_2)]$, where 
$q_1,\, q_2$ are any non-degenerate forms. For $p\equiv 3 \,(\mathrm{mod}\,4)$ we have $\W_{pt}(p)\cong \BZ/4\BZ$ 
and the class $[\C(\BZ/p\BZ,\,q)]$ is a generator for any non-degenerate form $q$. For 
$p\equiv 1 \,(\mathrm{mod}\, 4)$
the group $\W_{pt}(p)$ is isomorphic to $\BZ/2\BZ \oplus \BZ/2\BZ$;  it is generated by the two classes 
$[\C(\BZ/p\BZ,\,q')]$ and $[\C(\BZ/p\BZ,\,q'')]$  with $q'(l)=\zeta^{l^2}$ and $q''(l)=\zeta^{nl^2}$, where
$\zeta$ is a primitive $p$th root of unity in $\kk$ and 
$n$ is any quadratic non-residue modulo $p$.

\subsection{Property S} \label{propS}
Let $\C$ be a non-degenerate braided fusion category.

\begin{definition}
We say that $\C$ has {\em property S} if the following conditions
are satisfied:
\begin{enumerate}
\item[(S1)] $\C$ is completely anisotropic;
\item[(S2)] $\C$ is simple (that is, $\C$ has no non-trivial fusion
subcategories) and not pointed (so in particular
$\C \ncong \Vec$).
\end{enumerate}
\end{definition}

We will also say that a class $w\in \W$ has property S if a completely anisotropic representative
of $w$ has
property S. In Section~ \ref{sl2} we will give infinitely many examples of non-degenerate braided fusion
categories with property S.



\begin{theorem} \label{prS}
Let $\D=\boxtimes_{i\in I}\, \C_i$ where $\C_i$ are braided fusion categories
with property S. Assume that $\D$ is a Drinfeld center of a fusion category.
Then there is a fixed point free involution $a: I\to I$ such that
$\C_{a(i)}\simeq \C_i^{\rev}$
\end{theorem}
\begin{proof} 
Assume that $\D=\Z(\A)$ for some fusion category $\A$. Let $F: \D=\Z(\A)\to \A$ be the
forgetful functor. Choose a bijection $I= \{1,\dots,n\}$. For $1\le i\le n$ let $\A_i$ be the image of
$\C_1\boxtimes \C_2\boxtimes \cdots \C_i$ under $F$ (so $\A_i$ is a fusion subcategory of $\A$).

{\bf Claim:} {\em There is a subset $J_i\subset \{1,\dots,i\}$ such that $F$ restricted to
$\boxtimes_{j\in J_i}\,\C_j\subset \D$ is an equivalence $\boxtimes_{j\in J_i}\,\C_j\simeq \A_i$.}

{\em Proof (of the claim).} We use induction in $i$.
For $i=1$ we set $J_1=\{ 1\}$; in this case the claim follows from Lemma \ref{caninj}.
Now consider the induction step. The subcategory $\A_{i+1}$ is clearly generated by $\A_i$ and
(the image of) $\C_{i+1} \subset \A$ (recall that by Lemma \ref{caninj}, the functor $F$ restricted to
 $\C_{i+1}$ is fully faithful). There are two possibilities:

(a) the subcategories $\A_i$ and $\C_{i+1}$ intersect non-trivially in $\A$; then $\A_i$
contains $\C_{i+1}$ since by (S2) $\C_{i+1}$ has no non-trivial subcategories. In this case we
set $J_{i+1}=J_i$.

(b) $\A_i$ and $\C_{i+1}$ intersect trivially. Then we set $J_{i+1}=J_i\cup \{ i+1\}$.
We claim that the forgetful functor $\boxtimes_{j\in J_{i+1}}\,\C_j\to \A$ is fully faithful.
As in the proof of Lemma \ref{caninj} it is sufficient to show that for any
object $Z\in \boxtimes_{j\in J_{i+1}}\C_j$ we have $\Hom_\A(F(Z),\be)=\Hom_\D(Z,\be)$. Clearly,
we can restrict ourselves to the case when $Z$ is simple. In this case $Z=X\boxtimes Y$ where
$X\in \boxtimes_{j\in J_{i}}\C_j$ and $Y\in \C_{i+1}$ are simple. Then $F(Z)=F(X)\ot F(Y)$ where
$F(X)\in \A_i$ and $F(Y)\in F(\C_{i+1})$ are simple. Then
$\Hom_\A(F(Z),\be)=\Hom_\A(F(X),\,F(Y)^*)=0$ unless $X=\be$ and $Y=\be$. We are done in this case
and the claim is proved.

We apply now the Claim with $i=n$; we see that $\A=\boxtimes_{j\in J_n}\C_j$.
Thus $\Z(\A)=\boxtimes_{j\in J_n}(\C_j\boxtimes \C_j^{\rev})$ (see Section~ \ref{dcenter}).
The category $\D$ does not contain non-trivial invertible objects. By Proposition~\ref{needed reference}
it has a unique decomposition into a product of simple categories.
The result follows.
\end{proof}


\begin{corollary} Let $\C$ be a category with property S. Then $[\C]\in \W$ has order 2
if $\C \simeq \C^{\rev}$ and otherwise $[\C]\in \W$ has infinite order. \qed
\end{corollary} 

More precisely we have the following result. Let $\cS$ be the set of braided equivalence classes of
categories with property S. Let $\cS_2\subset \cS$ be the subset consisting of categories $\C$ such
that $\C \simeq \C^{\rev}$ and let $\cS_\infty =\cS \setminus \cS_2$. It is clear that the set $\cS$ is
at most countable, see Remark \ref{Wcount}. It follows from \eqref{sl2k+} in Section~\ref{sl2} below that
the set $\cS_\infty$ (and hence $\cS$) is infinite. Let $\cS_\infty'\subset \cS_\infty$ be a maximal
subset such that $\C \in \cS_\infty'$ implies $\C^\rev \not \in \cS_\infty'$

\begin{corollary} \label{ws}
Let $\W_S \subset \W$ be the subgroup generated by the categories with property S.
The map $(a_i)_{\C_i \in \cS}\mapsto \prod_{\C_i \in \cS}[\C_i]^{a_i}$ defines an isomorphism
$$\bigoplus_{\cS_2}\BZ/2\BZ \oplus \bigoplus_{\cS_\infty'}\BZ \simeq \W_S. \qed $$
\end{corollary} 

\begin{remark} 
1. It is clear that the set $\cS_2$ is at most countable. However we don't know whether it is empty and 
we don't know whether it is finite.

2. The description of the group $\W_S$ above is non-canonical due to the choice of the set $\cS_\infty'$.
A better description is as follows: the set $\cS$ carries an involution $\sigma$ which sends $\C$ to
$\C^{\rev}$. We extend $\sigma$ to the involution of the free abelian group $\BZ[\cS]$ generated
by $\cS$ by linearity. Then $\W_S\simeq \BZ[\cS]/\text{Image}(1+\sigma)$.

3. An argument similar to the proof of Theorem \ref{prS} shows that $\W_S\cap \W_{pt}=\{ 1\}$. 
Thus the subgroup of $\W$ generated by $\W_S$ and $\W_{pt}$ is isomorphic to $\W_S\times \W_{pt}$.


4. Assume that $\C_i$ are braided fusion categories with property S and $\C_i^{\rev}\not \simeq \C_j$ 
for $j\ne i$. Corollary \ref{ws} implies that $[\boxtimes_{i\in I}\, \C_i]\ne 0$. A stronger statement is true:
the category $\D=\boxtimes_{i\in I}\, \C_i$ is completely anisotropic. Indeed, by Lemma \ref{eacfl}
it is sufficient to show that any surjective central functor $\D \to \A$ is an equivalence. This is proved
by an argument parallel to the proof of Theorem \ref{prS}; notice that the case (a) in the proof of 
the Claim never occurs since otherwise we would have a non-injective central functor
$\C_i\bt \C_j\to \A$; considering the image of this functor one shows that $\C_i^\rev \simeq \C_j$ 
as in the proof  of Theorem \ref{unica}.
\end{remark}

\begin{corollary} 
The $\BQ-$vector space $\W \ot_\BZ \BQ$ also has countable infinite dimension.
\end{corollary}
\begin{proof}
Since $\cS_\infty$ is infinite, the $\BQ-$vector space $\W_S \ot_\BZ \BQ$ has countable infinite dimension. The result follows since the functor $?\ot_\BZ \BQ$ is exact.
\end{proof}

\subsection{Central charge} 
From now on we will assume that $\kk =\CC$. 
Recall that  any pseudo-unitary
non-degenerate braided fusion category has a natural structure of modular tensor category  (see, e.g.,
\cite[Section~ 2.8.2]{DGNO}). 

\begin{definition} 
\label{Wun}
Let $\W_{un}\subset \W$ be the subgroup consisting of Witt classes $[\C]$ of
pseudo-unitary non-degenerate braided fusion categories $\C$.
\end{definition}

\begin{remark}
Note that $\W_{un}$ is {\em not} invariant under the Galois action from Remark \ref{galois} (for
example class $[\C(sl(2),3)_+]\in \W_{un}$ from Section \ref{sl2} below  has a Galois conjugate not 
lying in $\W_{un}$). In particular, $\W_{un} \subsetneq \W$.
\end{remark}

Now recall that for a modular tensor category $\C$ one defines the {\em multiplicative central charge}
$\xi(\C)\in \CC$, see \cite[Section~ 6.2]{DGNO}. The following properties are well known, see, e.g, 
\cite[Section~ 3.1]{BaKi}.

\begin{lemma} \label{xi}
\begin{enumerate}
\item[(i)] $\xi(\C)$ is a root of unity;
\item[(ii)] $\xi(\C_1\bt \C_2)=\xi(\C_1)\xi(\C_2)$;
\item[(iii)] $\xi(\C^{\rev})=\xi(\C)^{-1}$. \qed
\end{enumerate}
\end{lemma}

The statement (i) (due to Anderson, Moore and Vafa, see \cite{AM, V}) allows us to consider the 
{\em additive central charge} $c=c(\C)\in \BQ/8\BZ$, which is related to $\xi(\C)$ by
$\xi(\C)=e^{2\pi ic/8}$.

\begin{lemma} \label{welldef}
Let $\C_1$ and $\C_2$ be two pseudo-unitary non-degenerate braided fusion
categories considered as modular tensor categories. Assume that $\C_1$ and $\C_2$ are Witt
equivalent. Then $\xi(\C_1)=\xi(\C_2)$.
\end{lemma}

\begin{proof} By Corollary \ref{easywitt} $\C_1\bt \C_2^{\rev}\simeq \Z(\A)$. Since the category 
$\C_1\bt \C_2^{\rev}$ is pseudo-unitary so is $\A$ (use \eqref{dimZ}). 
Thus, the spherical structure on $\C_1\bt \C_2^{\rev}=\Z(\A)$ is
induced by the spherical structure on $\A$. In this situation \cite[Theorem 1.2]{Mu4} says
that $\xi(\Z(\A))=1$. The result follows from Lemma \ref{xi}.
\end{proof}

Now for any class $w\in \W_{un}$ we define $\xi(w)=\xi(\C)$ where $\C$ is a pseudo-unitary
representative of the class $w$; according to Lemma \ref{welldef} this is well defined.
Similarly, we set $c(w)=c(\C)$. 

\begin{corollary} 
The assignment $w\mapsto c(w)$ is a homomorphism $\W_{un}\to \BQ/8\BZ$. 
\end{corollary}
\begin{proof}
This is immediate from Lemma \ref{xi}.
\end{proof}

\begin{remark}
A non-degenerate pointed category $\C(A,\,q)$ has a canonical pseudo-unitary structure (characterized by the condition that dimensions of all simple objects are 1). The ribbon twist of the corresponding modular structure on $\C(A,\,q)$ is $\theta_{X_a} = q(a)1_{X_a}$, where $X_a$ is a simple object corresponding to $a\in A$. 
The multiplicative central charge of  $\C(A,\,q)$ is given by \cite[Section~ 6.1]{DGNO}
$$\xi(\C(A,\,q)) = \frac{1}{\sqrt{|A|}}\sum_{a\in A}q(a).$$

In particular, for a metric cyclic group of order 2 with the value of the quadratic form on the generator $q(1)= i\in k$ (with $i^2=-1$) we have
$$\xi(\C(\BZ/2\BZ,\,q)) = \frac{1+i}{\sqrt{2}}$$ so that the additive central charge is
\begin{equation}\label{pcc}
c(\C(\BZ/2\BZ,q)) =  1\in\BQ/8\BZ.
\end{equation}
\end{remark}


\section{Finite extensions of vertex algebras} \label{vertex}

\subsection{Extensions of VOAs} \label{vex}
 
Let $V$ be a rational vertex algebra, that is
vertex algebra satisfying conditions 1-3 from \cite[Section~ 1]{Hu}. It is proved in {\em loc.\ cit.} that the
category $\Rep(V)$ of $V-$modules of finite length has a natural structure of modular tensor category; 
in particular
$\Rep(V)$ is a non-degenerate braided fusion category. 

Note that a rational vertex algebra has to be {\em simple} (i.e., have no non-trivial ideals). This, in particular, means that VOA maps
between rational vertex algebras are monomorphisms.

The category of modules $\Rep(V\otimes U)$ of the tensor product of two (rational) vertex algebras is ribbon equivalent
to the tensor product $\Rep(V)\boxtimes\Rep(U)$ of the categories of modules (see, for example \cite{fhl}).

The following relation between the central charge $c_V$ of a (unitary) rational VOA $V$ and the central charge of the category of its
modules $\Rep(V)$ is well-know to specialists (although we could not find a reference)\footnote{This relation can be verified directly for all the examples we consider later.}
: $$\xi(\Rep(V)) = e^{\frac{2\pi ic_V}{8}}.$$

Now consider a finite extension of vertex
algebras $V\subset W$, that is $V$ is a vertex subalgebra of $W$ (with the same Virasoro vector) and $W$ viewed as a 
$V-$module decomposes into a finite direct sum of irreducible $V-$modules 
\footnote{Note that finiteness is automatic if we assume that $L_0$-eigenspaces are finite dimensional (which is standard
and true e.g. for affine VOAs). Indeed, as a module over a rational vertex algebra $V$, $W$ is completely reducible, i.e. is a sum of simple $V$-modules. Since $V$ has only a finite number of non-isomorphic simple modules the only way for $W$ not to be finite is to have infinite multiplicities (in decomposition into simple $V$-modules). That will contradict finite dimensiality of $L_0$-eigenspaces.}.
Then $W$ considered
as an object $A\in \Rep(V)$ has a natural structure of commutative algebra; moreover this algebra
satisfies the conditions from Example \ref{regularce} (ii) and hence is \'etale, see 
\cite[Theorem 5.2]{KiO}\footnote{The proof of this result in \cite{KiO} is not complete. However for examples we are 
going to consider in this section the arguments from \cite[\S 5.5]{KiO} are sufficient.}.
Furthermore, the restriction functor $\Rep(W)\to \Rep(V)$ induces a braided tensor equivalence
$\Rep(W)\simeq \Rep(V)_A^0$. Thus, Proposition~\ref{rep0} implies that in this situation we have
$[\Rep(V)]=[\Rep(W)]$. We can use this in order to construct examples of interesting relations
in the group $\W$.

\begin{example} (Chiral orbifolds.)
Let $G$ be a finite group of automorphisms of a rational vertex algebra $V$. The sub-VOA of invariants $V^G$ is called the {\em
chiral orbifold} of $V$. In the case when the vertex subalgebra of invariants $V^G$ is rational we have a Witt equivalence
between categories of modules $\Rep(V)$, $\Rep(V^G)$.
\end{example}

\subsection{Affine Lie algebras and conformal embeddings} \label{confemb}

Let $\g$ be a finite dimensional simple Lie algebra and let $\hat \g$ be the
corresponding affine Lie algebra. For any $k\in \BZ_{>0}$ let $\C(\g,k)$ be the
category of highest weight integrable $\hat \g-$modules of level $k$, see e.g. \cite[Section~ 7.1]{BaKi} 
where this category is denoted $\O_k^{int}$. The category $\C(\g,k)$ can be identified with
the category $\Rep(V(\g,k))$ where $V(\g,k)$ is the simple
vertex algebra associated with the vacuum 
$\hat \g-$module of level $k$. In particular the category $\C(\g,k)$ has a structure of 
modular tensor category, see \cite{HuL}, \cite[Chapter 7]{BaKi}. 

\begin{example} The category $\C(sl(n),1)$ is pointed. 
It identifies with $\C(\BZ/n\BZ,q)$ where $q(l)=e^{\pi il^2\frac{n-1}n}$, $l\in \BZ/n\BZ$.
More generally, $\C(\g,1)$ (with $\g$ simply laced) is pointed \cite{fk}. 
\end{example}

It is known \cite{BaKi} the categories $\C(\g,k)$ are pseudo-unitary. In particular, we have
Witt classes $[\C(\g,k)]\in \W_{un}\subset \W$. The following formula for the central
charge is very useful, see e.g. \cite[7.4.5]{BaKi}:
\begin{equation} \label{cformula}
c(\C(\g,k))=\frac{k\dim \g}{k+h^\vee},
\end{equation}
where $h^\vee$ is the dual Coxeter number of the Lie algebra $\g$.

One can construct examples of relations between the classes $[\C(\g,k)]$ using 
the theory of {\em conformal embeddings}, see \cite{bb, SW, KW}.
Let $\bigoplus_i\g^i \subset \g'$ be an embedding (here $\g^i$ and $\g'$ are finite dimensional
simple Lie algebras). We will symbolically write $\oplus_i(\g^i)_{k_i}\subset \g'_{k'}$ if the restriction of
a $\hat \g'-$module of level $k'$ to $\hat \g^i$ has level $k_i$ (in this case the numbers $k_i$ are multiples of $k'$). 
Such an embedding defines an embedding of vertex algebras
$\otimes_iV(\g^i,k_i)\subset V(\g',k')$; but in general this embedding does not preserve 
the Virasoro vector. In the case when it does the embedding $\oplus_i(\g^i)_{k_i}\subset \g'_{k'}$
is called conformal embedding; it is known that in this case the extension of vertex algebras
$\otimes_iV(\g^i,k_i)\subset V(\g',k')$ is finite\footnote{This follows from the fact that $L_0$-eigenspaces of $V(\g,k)$ are finite dimensional}. 
Thus in view of Section~ \ref{vex}, we get a relation
\begin{equation}
\prod_i [\C(\g^i,k_i)]=[\C(\g',k')].
\end{equation}
The complete classification of the conformal embeddings was done in \cite{bb, SW} (see also \cite{KW}) and is reproduced in the Appendix. 

\subsection{Cosets}\label{cosets}

Let $U\subseteq V$ be an embedding of rational vertex algebras, which does not preserve conformal vectors $\omega_U,\omega_V$
(only operator products are preserved). The {\em centralizer} $C_V(U)$ is a vertex algebra with the conformal vector
$\omega_V-\omega_U$, see \cite{gko}. 
Moreover the tensor product $U\otimes C_V(U)$ is mapped naturally to $V$ and this map is a
map of vertex algebras. In the case when $V,U$ and $C_V(U)$ are rational we have a Witt equivalence of categories of modules:
$$\Rep(U)\boxtimes\Rep(C_V(U))\simeq \Rep(U\otimes C_V(U))$$ and $\Rep(V)$.

Let $\bigoplus_i(\h^i)_{k_i} \subset \bigoplus_j(\g^j)_{k'_j}$ be an embedding of 
vertex algebras non necessarily preserving the Virasoro vector as in Section \ref{confemb}. Let
$\otimes_iV(\h^i,k_i)\subset \otimes_jV(\g^j,k'_j)$ be the corresponding embedding of the
vertex algebras. The centralizer 
\[
C_{\otimes_jV(\g^j,k'_j)}(\otimes_iV(\h^i,k_i))
\]
is called 
the {\em coset} model and is denoted $\frac{\times_j(\g^j)_{k'_j}}{\times_i(\h^i)_{k_i}}$. 

Sometimes coset models defined by different embeddings of semisimple Lie algebras
are isomorphic. An example of such isomorphism was found by Goddard,
Kent and Olive \cite{gko}. They observed that the following coset 
models\footnote{here and in the Appendix the notation $X_{i,k}$ refers to the Lie algebra of type $X_i$
at level $k$.}:
$$\frac{A_{1,m}\times A_{1,1}}{A_{1,m+1}},\quad \frac{C_{m+1,1}}{C_{m,1}\times C_{1,1}}$$ are isomorphic, since they are both
isomorphic to the same
rational Virasoro vertex algebra $Vir_{c_m}$ with the central charge
\begin{equation}\label{mcc}
c_m=1-\frac{6}{(m+2)(m+3)}.
\end{equation}

We can use coset models in order to construct new relations in the Witt group as follows.
Assume that the central charge $c$ of a coset model vertex algebra 
$\frac{\times_j(\g^j)_{k'_j}}{\times_i(\h^i)_{k_i}}$ is positive\footnote{It is known (see \cite{gko}) that
$c\ge 0$. The case $c=0$ corresponds exactly to the conformal embeddings discussed in Section \ref{confemb}.} but less than 1\footnote{The list of cosets with such central charge was given in \cite{bg} and is reproduced  in the Appendix.}. 
It is known that in this case $c=c_m$ for some positive integer $m$ and the vertex algebra in question
contains a rational vertex subalgebra $Vir_{c_m}$, see \cite{gko}. This implies that the rational vertex 
algebra $\otimes_jV(\g^j,k'_j)$ is a finite extension of rational vertex algebra
$\otimes_iV(\h^i,k_i)\otimes Vir_{c_m}$. Thus according to the results of Section \ref{vex}
we get a relation in the Witt group
\begin{equation}\label{coset1}
(\prod_i [\C(\h^i,k_i)])\cdot [Vir_{c_m}]=\prod_j[\C(\g^j,k'_j)].
\end{equation}
A special case of this relation corresponding to the coset model $\frac{A_{1,m}\times A_{1,1}}{A_{1,m+1}}$ reads
\begin{equation}\label{coset2}
[Vir_{c_m}]=[\C(sl(2),m)][\C(sl(2),1)][\C(sl(2),m+1)]^{-1}.
\end{equation}
Thus combining \eqref{coset1} and \eqref{coset2} we obtain relations between the classes $[\C(\g,k)]$.


\subsection{Examples for $\g=sl(2)$} \label{sl2}
We give here some examples of relations (or absence thereof) between the classes
$[\C(sl(2),k)]$. We refer the reader to \cite[Section~ 6]{KiO} for more details on the categories $\C(sl(2),k)$.
Note that all \'etale algebras in these categories were classified in \cite[Theorem 6.1]{KiO}.

\begin{enumerate}
\item[(1)] The category $\C(sl(2),1)$ is pointed, moreover $\C(sl(2),1)\simeq \C(\BZ/2\BZ,q_+)$ 
where $q_+(1)=i$. In particular, class $[\C(sl(2),1)]\in \W$ has order 8. 

\item[(2)] For any odd $k$, we have $\C(sl(2),k)\simeq \C(sl(2),k)_+\boxtimes \C(\BZ/2\BZ,q_\pm)$ where
$\C(sl(2),k)_+$ is the subcategory of ``integer spin" representations and
$q_\pm(1)=\pm i$ (see e.g. \cite[Lemma 6.6]{KiO}). The category $\C(sl(2),k)_+$ for an odd $k\ge 3$
has property S. Using \eqref{cformula} and \eqref{pcc} we get
$$c(\C(sl(2),k)_+)=\frac{3k}{k+2}+(-1)^{(k+1)/2}.$$
In particular, $2c(\C(sl(2),k)_+)\ne 0\in \BQ/8\BZ$, 
so 
\begin{equation}
\label{sl2k+}
\C(sl(2),k)_+\ncong \C(sl(2),k)_+^{\rev}. 
\end{equation}
This shows 
that the set $\cS_\infty$ from Section~ \ref{propS} is infinite.

Consider the category $\C(sl(2),3)_+$. The class
$[\C(sl(2),3)_+]\in \W$ is a simplest example of element of $\W$ of infinite order. We will say that 
a braided fusion category $\C$ is a {\em Fibonacci category} if the Grothendieck ring $K(\C)$ 
is isomorphic to $K(\C(sl(2),3)_+)$ as a based ring. It is known that a pseudo-unitary Fibonacci category 
is equivalent to either $\C(sl(2),3)_+$ or $\C(sl(2),3)_+^{\rev}$. 

\item[(3)] The category $\C(sl(2),2)$ is an example of Ising braided category, see \cite[Appendix B]{DGNO}.
In particular, it follows from \cite[Lemma B.24]{DGNO} that 
\[ 
[\C(sl(2),2)]^2=[\C(\BZ/4\BZ,q)], \quad \mbox{where} \quad
q(l)=e^{3\pi il^2/4}. 
\]
Thus, the order of $[\C(sl(2),2)]\in \W$ is 16. 

Using \cite[Lemma B.24]{DGNO} it is easy to see that for an odd $l$ we have $[\C(sl(2),2)]^l=[\C]$,
where $\C$ is an Ising braided category. Since there are precisely 8 equivalence classes of
Ising braided categories (see \cite[Corollary B.16]{DGNO}), we get that for any Ising braided category
$\C$ there is a unique odd number $l$, $1\le l\le 15$ such that $[\C]=[\C(sl(2),2)]^l$. The number
$l$ is easy to compute from $c(\C)$ using $c(\C(sl(2),2))=\frac32$.

\item[(4)] There exists a conformal embedding $sl(2)_4\subset sl(3)_1$. Thus 
\[
[\C(sl(2),4)]=[\C(sl(3),1)]=[\C(\BZ/3\BZ, q)], \quad \mbox{where} \quad  q(l)=e^{2\pi il^2/3}.
\] 
In particular, the order of $[\C(sl(2),4)]\in \W$ is 4.

\item[(5)] There exists a conformal embedding $sl(2)_6\oplus sl(2)_6\subset so(9)_1$. Thus
\[
[\C(sl(2),6)]^2=[\C(so(9),1)]. 
\]
Notice that $\C(so(9),1)$ is also an example of Ising braided category.
Using the central charge one computes that 
\[
[\C(sl(2),6)]^2=[\C(sl(2),2)]^3.
\] 
In particular, $[\C(sl(2),6)]\in \W$ has order 32.

\item[(6)] The category $\C(sl(2),8)$ is known to contain an \'etale algebra $A$ such that $\C(sl(2),8)_A^0$
is equivalent to the product of two Fibonacci categories, see e.g., \cite[Theorem 4.1]{MPS}. 
Using the central charge one computes that 
\[
[\C(sl(2),8)]=[\C(sl(2),3)_+]^{-2}.
\]
 
\item[(7)] There exists a conformal embedding $sl(2)_{10}\subset sp(4)_1$. Thus, 
\[
[\C(sl(2),10)]=[\C(sp(4),1)]. 
\]
The category $\C(sp(4),1)$ is an Ising braided category. 
Using the central charge one computes that $[\C(sl(2),10)]=[\C(sl(2),2)]^7$. 

\item[(8)] Let $\g(G_2)$ be a Lie algebra of type $G_2$. There exists a conformal embedding 
$sl(2)_{28}\subset \g(G_2)_1$. Thus,
\[
[\C(sl(2),28)]=[\C(\g(G_2),1)]. 
\]
The category $\C(\g(G_2),1)$ is
a Fibonacci category. Using the central charge one computes that 
\[
[\C(sl(2),28)]=[\C(sl(2),3)_+].
\]

\item[(9)] The category $\C(sl(2),k)$ with $k$ divisible by 4 is known to contain an \'etale algebra $A$
of dimension 2, see \cite[Theorem 6.1]{KiO}. It is also known that in this case for $k\ne 4,\,8,\,28$
the category $\C(sl(2),k)_A^0$ has property S and is not equivalent to any category $\C(sl(2),k_1)_+$
with odd $k_1$. Thus we get infinitely many more elements of the set $\cS_\infty$. For example
we see that $[\C(sl(2),12)]\in \W$ has infinite order.
\end{enumerate}

\subsection{Holomorphic vertex algebras with $c=24$} We recall that a rational vertex algebra $V$ is 
called {\em holomorphic} if $\Rep(V)=\Vec$, that is the only simple $V-$module is $V$ itself, 
see e.g. \cite{DM}. In \cite{Sc} Schellekens gives a conjectural list of 71 holomorphic vertex
algebras with central charge $c=24$, see also \cite{DM}. Out of this list, 69 algebras are extensions 
of vertex algebras associated with affine Lie algebras as in Section \ref{confemb}. Thus in view 
of discussion in Section \ref{vex}, each of these algebras should give a conjectural relation between 
the classes $[\C(\g,k)]$. Some of these relations can be deduced from the relations in 
Sections \ref{confemb} and \ref{cosets}, but some others are genuinely new. For example an entry 
No.14 from the Schellekens list gives a conjectural relation $[\C(F_4,6)]=[\C(sl(3),2)]^{-1}$ which 
can not be deduced from the results above.

\subsection{Open questions} In this section we collect some open questions about the Witt group $\W$.

\begin{question} Is it true that $\W$ is a direct sum of cyclic groups? Is there an inclusion
$\BQ \subset \W$?
\end{question}

\begin{question} Is $\W_{un}$ generated by classes $[\C(\g,k)]$?
\end{question}

\begin{remark} Notice that $\W_{pt}$ is contained in the subgroup generated by $[\C(\g,k)]$.
Namely, the subgroup of $\W$ generated by $[\C(sl(2),1)]$ and $[\C(sl(2),2)]$ contains
$\W_{pt}(2)$. For a prime $p=4k+3$, the subgroup $\W_{pt}(p)$ is generated by $[\C(sl(p),1)]$.
Finally for a prime $p=4k+1$ choose a prime number $q<p$ which is a quadratic non-residue
modulo $p$ (it is easy to see that such a prime does exist). Then $\W_{pt}(p)$ is contained
in the subgroup of $\W$ generated by $[\C(sl(p),1)]$ and $[\C(sl(pq),1)]$ and $\W_{pt}(q)$. Thus
we are done by induction.
\end{remark}

\begin{remark}
Since the end of eighties there is a common belief among physicists that all rational conformal 
field theories come from lattice and WZW models via coset and orbifold (and perhaps chiral extension) constructions (see \cite{ms}). 
Analogous statement for modular categories would imply that the unitary Witt group is generated by classes of affine categories $\C(\g,k)$. 
\end{remark}

\begin{question} What are the relations in the subgroup of $\W$ generated by $[\C(\g,k)]$?
Is it true that all relations in the subgroup generated by $[\C(sl(2),k)]$ are described in Section~ \ref{sl2}?
Is it possible to express some nonzero power of $[\C(sl(2),12)]\in \W$ in terms of $[\C(sl(2),k],\; k\ne 12$?
What is the order of $[\C(sl(2),14)]\in \W$?
\end{question}

\begin{question} Is there a class $w\in \W_S$ of order 2? Equivalently does exist a non-degenerate
braided fusion category $\C$ with property S and such that $\C^{\rev}\simeq \C$?
\end{question}

\begin{question}  Is it true that torsion in $\W$ is 2-primary? Is there an element of order 3 in $\W$?
\end{question}

\begin{question} What is the biggest finite order of an element of $\W$? For example, are there elements of $\W$
of order 64?
\end{question}

\section*{Appendix. Conformal embeddings and cosets with $c<1$}

Here we reproduce (from \cite{bb, SW} ) the list of maximal embeddings starting with serial embeddings
(rank-level dualities, (anti-)symmetric and regular embeddings and their variants) and followed up by sporadic embeddings. For the sake of compactness we use matrix algebra notations (instead of Dynkin symbols) for the rank-level embeddings (the first four).

$$\begin{array}{ll}
su(m)_n\times su(n)_m\subseteq su(mn)_1, & sp(2m)_n\times sp(2n)_m\subseteq so(4mn)_1, \\ so(m)_n\times so(n)_m\subseteq
so(mn)_1, & so(m)_4\times su(2)_m\subseteq sp(2m)_1,
\end{array}$$
$$\begin{array}{ll}
A_{n,n-1}\subseteq A_{\frac{(n-1)(n+2)}{2},1}, & A_{n,n+3}\subseteq A_{\frac{n(n+3)}{2},1},\\ A_{2n+1,2n+2}\subseteq
B_{2n^2+4n+1,1}, & A_{2n,2n+1}\subseteq D_{2n(n+1),1},\\ B_{2n+1,4n+1}\subseteq B_{(4n+1)(n+1),1}, & A_{2n+1,4n+5}\subseteq
B_{4n^2+7n+2,1},\\ B_{2n,4n-1}\subseteq D_{n(4n+1),1}, & B_{2n,4n+3}\subseteq D_{n(4n+3),1},\\ C_{2n,2n-1}\subseteq
B_{4n^2-n-1,1}, & C_{2n+1,2n+2}\subseteq B_{(4n+1)(n+1),1} \\ C_{2n,2n+1}\subseteq D_{n(4n+1),1}, & C_{2n+1,2n}\subseteq
D_{n(4n+3),1},\\ D_{2n,4n+2}\subseteq B_{4n^2+n-1,1}, & D_{2n+1,4n}\subseteq B_{n(4n+3),1},\\ D_{2n,4n-2}\subseteq
D_{n(4n-1),1},& D_{2n+1,4n+4}\subseteq D_{(n+1)(4n+1),1},\\ B_{n,2}\subseteq A_{2n,1},& D_{n,2}\subseteq A_{2n-1,1},\\
D_{1,1}\times A_{i,1}\times A_{n-i-1}\subseteq A_{n,1}, 1\leq i\leq n-2 & D_{1,}\times A_{n-1,1}\subseteq A_{n,1},\\ D_{1,1
}\times D_{n-1,1}\subseteq D_{n,1}, & D_{1, 1}\times A_{n-1,1}\subseteq D_{n,1},\\ A_{1,1}\times A_{1,1}\times
D_{n-2,1}\subseteq D_{n,1}, & D_{i,1}\times D_{n-i,1}\subseteq D_{n,1}, 3\leq i\leq n-3 \\A_{1,1}\times A_{1,1}\times
B_{n-2,1}\subseteq B_{n,1}, & D_{1,1 }\times B_{n-1,1}\subseteq B_{n,1}, \\ D_{i,1}\times B_{n-1,1}\subseteq B_{n,1}, 3\leq
i\leq n-2 & D_{i,1}\times B_{n-i,1}\subseteq B_{n,1} \\ A_{1,2}\times D_{n-1,1}\subseteq B_{n,1}, & D_{1,1 }\times
A_{n-1,2}\subseteq C_{n,1}, \\ D_{1, 1}\times D_{5,1}\subseteq E_{6,1}, & A_{1,1}\times A_{5,1}\subseteq E_{6,1}, \\
A_{2,1}\times A_{2,1}\times A_{2,1}\subseteq E_{6,1}, & D_{1, 1}\times E_{6,1}\subseteq E_{7,1},\\ A_{1,1}\times
D_{6,1}\subseteq E_{7,1}, & A_{7,1}\subseteq E_{7,1},\\ A_{2,1}\times A_{5,1}\subseteq E_{7,1}, & D_{8,1}\subseteq E_{8,1}, \\
A_{4,1}\times A_{4,1}\subseteq E_{8,1}, & A_{2,1}\times E_{6,1}\subseteq E_{8,1}, \\ A_{1,1}\times E_{7,1}\subseteq E_{8,1},
& A_{8,1}\subseteq E_{8,1},\\ A_{1,1}\times C_{3,1}\subseteq F_{4,1}, & G_{2,1}\times A_{1,8}\subseteq F_{4,1},\\
A_{1,3}\times A_{1,1}\subset G_{2,1}, & A_{2,2}\times A_{2,1}\subset F_{4,1},\\
G_{2,1}\times A_{2,2}\subseteq E_{6,1}, & A_{1,7}\times G_{2,2}\subseteq E_{7,1}, \\ A_{1,3}\times F_{4,1}\subseteq E_{7,1},
& G_{2,1}\times C_{3,1}\subseteq E_{7,1}, \\ A_{2,6}\times A_{1,16}\subseteq E_{8,1}, & G_{2,1}\times F_{4,1}\subseteq
E_{8,1},\\
\end{array}$$
$$\begin{array}{lll}
A_{1,10}\subseteq B_{2,1}, & A_{1,28}\subseteq G_{2,1}, & A_{2,9}\subseteq E_{6,1},\\ A_{2,21}\subseteq E_{7,1}, &
A_{5,6}\subseteq C_{10,1}, & A_{7,10}\subseteq D_{35,1}, \\ B_{2,12}\subseteq E_{8,1}, & B_{4,2}\subseteq D_{8,1}, &
C_{3,5}\subseteq C_{7,1}, \\ C_{4,1}\subseteq E_{6,1}, & C_{4,7}\subseteq D_{21,1}, & D_{5,4}\subseteq A_{15,1}, \\
D_{6,8}\subseteq C_{16,1}, & D_{8,16}\subseteq D_{64,1}, & E_{6,6}\subseteq A_{26,1}, \\ E_{6,12}\subseteq D_{39,1}, &
E_{7,12}\subseteq C_{28,1}, & E_{7,18}\subseteq B_{66,1}, \\ E_{8,30}\subseteq D_{124,1}, & F_{4,3}\subseteq D_{13,1}, &
F_{4,9}\subseteq D_{26,1}, \\ G_{2,3}\subseteq E_{6,1}, & G_{2,4}\subseteq D_{7,1}. &
\end{array}$$

Next we reproduce the list of cosets with central charge $0<c<1$ given in \cite{bg}:

$$\begin{array}{ll}
Vir_{c_n} = \frac{A_{1,1}\times A_{1,n}}{A_{1,n+1}}, & Vir_{c_n}\subseteq \frac{A_{n+1,2}}{A_{n,2}\times u(1)}, \\
Vir_{c_n}\subseteq \frac{C_{m+1,1}}{C_{m,1}\times C_{1,1}}, & Vir_{c_1}\subseteq \frac{so(n)_1}{so(n-1)_1},
\end{array}$$
$$\begin{array}{lll}
Vir_{c_1}\subseteq \frac{A_{1,2}}{u(1)} ,& Vir_{c_2}\subseteq \frac{E_{7,2}}{A_{7,2}}, & Vir_{c_3}\subseteq
\frac{A_{2,1}\times A_{2,1}}{A_{2,2}},\\ Vir_{c_3}\subseteq \frac{A_{1,3}}{u(1)} ,& Vir_{c_3}\subseteq
\frac{E_{7,2}}{A_{1,2}\times D_{6,2}}, & Vir_{c_4}\subseteq \frac{E_{6,1}\times E_{6,1}}{E_{6,2}},\\ Vir_{c_3}\subseteq
\frac{A_{2,2}}{A_{1,8}} ,& Vir_{c_1}\subseteq \frac{E_{8,2}}{D_{8,2}}, & Vir_{c_2}\subseteq \frac{E_{7,1}\times
E_{7,1}}{E_{7,2}},\\ Vir_{c_3}\subseteq \frac{A_{5,1}}{C_{3,1}} ,& Vir_{c_2}\subseteq \frac{E_{8,2}}{A_{1,2}\times
E_{7,1}}, & Vir_{c_1}\subseteq \frac{E_{8,1}\times E_{8,1}}{E_{8,2}},\\ Vir_{c_2}\subseteq \frac{B_{3,1}}{G_{2,1}} ,&
Vir_{c_9}\subseteq \frac{E_{8,2}}{A_{8,2}}, & Vir_{c_9}\subseteq \frac{E_{8,1}\times E_{8,2}}{E_{8,3}},\\
Vir_{c_3}\subseteq \frac{E_{6,1}}{F_{4,1}} ,& Vir_{c_9}\subseteq \frac{F_{4,1}}{C_{3,2}\times A_{1,2}}, &
Vir_{c_8}\subseteq \frac{F_{4,1}\times F_{4,1}}{F_{4,2}},\\ Vir_{c_4}\subseteq \frac{E_{6,2}}{C_{4,2}} ,&
Vir_{c_2}\subseteq \frac{F_{4,1}}{B_{4,1}}, & Vir_{c_7}\subseteq \frac{G_{2,1}\times G_{2,1}}{G_{2,2}},\\
Vir_{c_5}\subseteq \frac{E_{6,2}}{A_{1,2}\times A_{5,2}} ,& Vir_{c_3}\subseteq \frac{G_{2,1}}{A_{2,1}}, &
Vir_{c_6}\subseteq \frac{G_{2,2}}{A_{1,2}\times A_{1,6}}
\end{array}$$

\bibliographystyle{ams-alpha}

\end{document}